\documentclass[10pt,a4paper]{article}
\usepackage[latin1]{inputenc}
\usepackage{amsmath}
\usepackage{mathrsfs}
\usepackage{amsfonts}
\usepackage{amssymb}
\usepackage{latexsym}
\usepackage{yfonts}
\usepackage{amsthm}
\usepackage{makeidx}
\usepackage{fancyhdr}
\usepackage{natbib}
\bibpunct{[}{]}{;}{n}{,}{,}

\usepackage[margin=4cm]{geometry}

\allowdisplaybreaks

\usepackage{graphicx,color}

\usepackage[pdftex,plainpages=false,colorlinks,hyperindex,bookmarksopen,linkcolor=red,citecolor=blue,urlcolor=blue]{hyperref}

\author{Mirko D'Ovidio\footnote{Dipartimento di Scienze di Base e Applicate per l'Ingegneria, via A. Scarpa 10, 00161 Rome (Italy),  e-mail:mirko.dovidio@uniroma1.it} \\{\em Sapienza University of Rome}}

\title{From Sturm-Liouville problems\\ to fractional and anomalous diffusions}

\newtheorem{te}{Theorem}
\newtheorem{prop}{Proposition}
\newtheorem{df}{Definition}
\newtheorem{lem}{Lemma}
\newtheorem{os}{Remark}
\newtheorem{coro}{Corollary}

\renewcommand{\H}{H^{m,n}_{p,q}\left[ x \bigg| \begin{array}{l} (a_i, \alpha_i)_{i=1, .. , p}\\ (b_j, \beta_j)_{j=1, .. , q}  \end{array} \right]}

\numberwithin{equation}{section}

\begin{document}

\maketitle
\begin{center}
\begin{minipage}{4in}
\textbf{Abstract} Some fractional and anomalous diffusions are driven by equations involving fractional derivatives in both time and space. Such diffusions are processes with randomly varying times. In representing the solutions to those equations, the explicit laws of certain stable processes turn out to be fundamental. This paper directs one's efforts towards the explicit representation of solutions to fractional and anomalous diffusions related to Sturm-Liouville problems of fractional order associated to fractional power function spaces. Furthermore, we study a new version of the Bochner's subordination rule and we establish some connections between subordination and space-fractional operator.
\end{minipage}
\end{center}


\section{Introduction and main results}

In recent years, many researchers have shown their interest in fractional and anomalous diffusions. The term fractional is achieved by replacing standard derivatives w.r.t.\ time $t$ with fractional derivatives, for instance, those of Riemann-Liouville or Dzhrbashyan-Caputo. Anomalous diffusion occurs, according to most of the significant  literature, when the mean square displacement (or time-dependent variance) is stretched by some index, say $\alpha \neq 1$ or, in other words proportional to a power $\alpha$ of time, for instance $t^\alpha$. Such anomalous feature can be found in transport phenomena in complex systems, e.g.\ in random fractal structures (see \citet{GioRom92}).

Fractional diffusions have been studied by several authors. \citet{Wyss86}, \citet{SWyss89} and later \citet{Hil00} studied the solutions to the heat-type fractional diffusion equation and presented such solutions in terms of Fox's functions. For the same equation, up to some scaling constant, \citet{OB03,OB09} represented the solutions by means of stable densities and found the explicit representations only in some cases. Different boundary value problems have also been studied by \citet{MezK00, BO09elastic}. In the papers by \citet{MLP01, MPS05, MPG07} the authors presented the solutions to space-time fractional equations by means of Wright functions or Mellin-Barnes integral representations, that is Fox's functions. See also \citet{MMP10} for a review on the Mainardi-Wright function.

For a general operator $\mathscr{A}$ acting on space, several results can also be listed. \citet{Nig86} gave a physical interpretation when $\mathscr{A}$ is the generator of a Markov process whereas \citet{Koc89, Koc90} first introduced a mathematical approach. \citet{Zasl94} introduced the fractional kinetic equation for Hamiltonian chaos. \citet{BM01} studied the problem when $\mathscr{A}$ is an infinitely divisible generator on a finite dimensional space. For a short survey of these results see \citet{NANERW}.

In general, the stochastic solutions to fractional diffusion equations can be realized through subordination. Indeed, for a guiding process $X(t)$ with generator $\mathscr{A}$ we have that $X(V(t))$ is governed by $\partial^\beta_t u= \mathscr{A} u $ where the process $V(t)$, $t>0$ is an inverse or hitting time process to a $\beta$-stable subordinator (see \citet{BM01}). The time-fractional derivative comes from the fact that $X(V(t))$ can be viewed as a scaling limit of continuous time random walk where the iid jumps are separated by iid power law waiting times (see \citet{MSheff04, MezK00secondo, RomAle94}). Results on the subordination principle for fractional evolution equations can also be found in \citet{Bazh00,Boch49}.

Anomalous diffusions can also be carried out by considering a fractional operator acting on the space. The problem of finding a suitable representation for a fractional power of a given operator $A$ defined in a Banach space $X$ has a long history.  The first definitions of fractional power of the Laplace operator were introduced by \citet{Boch49, Feller52}. For a closed linear operator $A$, the fractional operator $(-A)^\alpha$ has been investigated by many researchers, see e.g. \citet{Balak60, HoWe72, Kom66, KraSob59, Wata61}. Although the methods presented differ, each of those papers was primarily based on the integral representation
\[ -(-A)^\alpha \, f= \frac{\sin \pi \alpha}{\pi}  \int_0^\infty \lambda^{\alpha -1}(\lambda I - A )^{-1}\, A\, f\, d\lambda \]
for a well defined $f$ and $0 < \Re\{ \alpha \} < 1$ under
\begin{equation}
\begin{array}{rl} (i) & \lambda \in \rho(A) \, (\textrm{the resolvent set of  } A) \textrm{ for all } \lambda >0;\\
(ii) & \| \lambda (\lambda I - A)^{-1} \| < M < \infty \textrm{ for all }\lambda >0.  \end{array}
\label{cvbA}
\end{equation}
Different definitions can be also given by means of hyper singular integrals, see e.g. \citet{Samko00}.

In both cases, time and space fractional equations, the explicit representations of the law of stable processes and, those of the corresponding inverse processes, are fundamental in finding explicit solutions to fractional problems. \\

In this paper we study time and space fractional problems involving the operator $\mathcal{G}^{*}$ (see formula \eqref{OrdOpG}) which is the adjoint of an infinitesimal generator of non-negative diffusions $\mathcal{G}$. In particular, for $\mathfrak{w}(x)=x^{\gamma \mu -1}$, the second order differential operator
\begin{align*}
\mathcal{G} = & \frac{1}{\gamma^2\, \mathfrak{w}(x)} \frac{\partial}{\partial x} x^{\gamma\mu -\gamma +1} \frac{\partial}{\partial x} , \quad \gamma=\pm 1,\; \mu >0
\end{align*}
is the operator governing two related diffusions, the squared Bessel process $G_\mu(t)$, $t>0$ (for $\gamma=+1$) and its inverse process $E_\mu(t)$, $t>0$ (also known as reciprocal process of $G_\mu$, for $\gamma=-1$). Due to the fact that $P\{ E_\mu(t)<x\} = P\{ G_\mu(x)>t \}$ we refer to $E_\mu$ as the inverse of $G_\mu$.

For such processes we study the governing equations where the time derivative is replaced by its fractional counterpart and find solutions in bounded and unbounded domains. For the time-fractional equations on bounded domains we study Sturm-Liouville problems of fractional order associated with fractional power function spaces. A complete orthogonal  set of eigenfunctions (w.r.t. the weight function $\mathfrak{w}(x)$) arises naturally as solutions of the second-order differential equations involving $\mathcal{G}$ under appropriate boundary conditions and therefore, we obtain Sturm-Liouville boundary-value problems associated to the killed semigroups of $G_\mu$ and $E_\mu$. 

The fractional power of $\mathcal{G}^{*}$ (for $\gamma=+1$, that is the governing operator of the squared Bessel process $G_\mu$) is also examined. We find the explicit representation of $-(-\mathcal{G}^{*})^{\nu}$ for $\nu \in (0, 1)$ in terms of Riemann-Liouville derivatives and we discuss the properties of the corresponding subordinated process. Thus, we obtain a representation of the power $\nu$ of the composition of non commuting operators (formula \eqref{OrdOpG} below).

All fractional problems investigated in this work have stochastic solutions with randomly varying times which are subordinators. Such subordinators are denoted by $\mathfrak{H}^\nu_t$ and $\mathfrak{L}^\nu_t$. From the fact that $P\{ \mathfrak{L}^\nu_t < x \} = P\{ \mathfrak{H}^\nu_x > t\}$ we say that $\mathfrak{L}^\nu_t$ is the inverse of $\mathfrak{H}^\nu_t$ which is a positively skewed stable process with non-negative increments and therefore non-decreasing paths. This means that $\mathfrak{L}^\nu_t = \inf\{ x \geq 0\,:\, \mathfrak{H}^\nu_x \notin (0, t) \}$ can be regarded as hitting time. We find that, for $\nu=1/n$, $n \in \mathbb{N}$,
\begin{equation}
\mathfrak{H}^{\nu}_t \stackrel{law}{=} E_{\mu_1}(E_{\mu_2}(\ldots E_{\mu_n}((\nu t)^{1/\nu})\ldots)), \quad t>0\label{comp1}
\end{equation}
and 
\begin{equation}
\mathfrak{L}^{\nu}_t \stackrel{law}{=} \left[ G_{\mu_1}(G_{\mu_2}(\ldots G_{\mu_n}(t^\nu/\nu)\ldots)) \right]^{\nu}, \quad t>0\label{comp2}
\end{equation}
for suitable choices of $\boldsymbol{\mu} = (\mu_1, \ldots , \mu_n)$. Furthermore, we show that the compositions \eqref{comp1} and \eqref{comp2} hold for all $\boldsymbol{\mu} \in \mathscr{P}^{n}_{n+1}(n!) $ where
\begin{align*}
\mathscr{P}_{\kappa}^{n}\left( \varrho \right) = & \left\lbrace \bar{\wp} \in \mathbb{R}^{n}_{+}\,:\,  \bar{\wp} = \frac{\bar{\upsilon}}{\kappa},\, \bar{\upsilon} = (\upsilon_1 , \ldots , \upsilon_n) \in \mathbb{N}^{n}, \, \prod_{j=1}^{n} \upsilon_j = \varrho  \right\rbrace
\end{align*}
with $m, \kappa, \varrho \in \mathbb{N}$. This result permit us to explicitly write the laws of $\mathfrak{H}^{1/n}_t$ and $\mathfrak{L}^{1/n}_t$ and therefore the laws of the processes subordinated by them. In particular, we obtain useful representations of solutions to fractional equations involving general operators but representing anomalous diffusions realized through subordination.\\

The main results of this work are collected in Section \ref{sezFAD}. We present auxiliary results and proofs in the remaining sections of the paper.

\section{Introductory remarks and notations}
\label{intNot}
We first introduce the following notation:
\begin{itemize}
\item $s_\nu$ is the law of the symmetric stable process $S^\nu_t$,
\item $h_\nu$ is the law of the stable subordinator $\mathfrak{H}^\nu_t$,
\item $l_\nu$ is the law of $\mathfrak{L}^\nu_t$ which is the inverse to $\mathfrak{H}^\nu_t$,
\item $g_\mu = g^1_\mu$ is the law of the squared Bessel process starting from zero $G_\mu(t)$,
\item $e_\mu = g^{-1}_\mu$ is the law of $E_\mu(t)$ which is the (reciprocal of ) inverse to $G_\mu(t)$,
\item $H^{m,n}_{p,q}$ is the Fox's function,
\item $W^\alpha_{+}$ and $W^\alpha_{-}$ are the left and right Weyl's derivatives,
\item $\frac{\partial^\alpha}{\partial x^\alpha}$ and $\frac{\partial^\alpha}{\partial (-x)^\alpha}$ are the left and right Riemann-Liouville derivatives,
\item ${^C D_{t}^\alpha}$ is the Dzhrbashyan-Caputo fractional derivative.
\end{itemize}

We now introduce fractional derivatives and recall their connection with stable densities. The $\alpha$-stable process $S^{\alpha, \theta}_t$, $t>0$, with law $s^{\theta}_{\nu}=s^{\theta}_{\nu}(x,t)$, $x \in \mathbb{R}$, $t>0$ has characteristic function
\begin{equation}
E \exp\left( i \,\beta\, S^{\alpha, \theta}_t \right) = \exp\Bigg[ -t |\xi |^\alpha \Big[ 1- i \theta \frac{\xi}{|\xi |} \tan\left( \frac{\alpha \pi}{2} \right) \Big] \Bigg],\quad \xi \in \mathbb{R}
\label{tauFourierT}
\end{equation}
with $\alpha \in (0,1) \cup (1,2]$ and $\theta \in [-1, 1]$ (see \citet{Zol86}). If $\theta=0$, then we have a symmetric process  with $E \exp\left( i \beta S^{\alpha}_t \right)  = \exp\left( -t|\beta |^\alpha  \right)$, $\alpha \in (0, 2]$. For the sake of simplicity we will write $S^{\alpha}_t$ instead of $S^{\alpha, 0}_t$. Moreover, we will refer to $\mathfrak{H}^{\nu}_t = S^{\nu, 1}_t$, $t>0$ as the totally (positively) skewed process which is also named stable subordinator. For $n-1 < \alpha < n$, $n \in \mathbb{N}$, according to \citet{KST06, SKM93}, we define
\begin{equation}
(W^{\alpha}_{-}f)(x) =  \frac{(-1)^n}{\Gamma(n-\alpha)} \frac{d^n}{dx^n} \int_x^{\infty} (s - x)^{n-\alpha -1} f(s) \, ds, \quad x \in \mathbb{R} \label{Weyll1}
\end{equation}
and
\begin{equation}
(W^{\alpha}_{+}f)(x) = \frac{1}{\Gamma(n-\alpha)} \frac{d^n}{dx^n} \int_{-\infty}^x (x -s)^{n-\alpha -1} f(s) \, ds, \quad x \in \mathbb{R} \label{Weyll2}
\end{equation}
which are the right and left Weyl's derivatives by means of which we write the governing equation of $S^{\alpha, \theta}_t$, $t>0$, given by
\begin{equation}
\frac{\partial s_{\alpha}}{\partial t}(x,t) =\, _\theta D^{\alpha}_{|x|} s_{\alpha}(x,t), \quad x \in \mathbb{R}, \; t>0 \label{RieszEq}
\end{equation}
where 
\begin{equation}
{_\theta D^{\alpha}_{|x|}} =  \frac{1}{2 \cos \alpha\pi/2} \Big[ \kappa \, W^{\alpha}_{-} + (1-\kappa) \, W^{\alpha}_{+} \Big]  \label{weylrap}
\end{equation}
and $0 \leq \kappa=\kappa(\theta) \leq 1$ (see e.g. \citet{BWM00, Ch98, MLP01, OZXUE03}). For $\theta=0$ (that is $\kappa= \kappa(0)=1/2$) we obtain the Riesz derivative
\begin{equation}
{ _0 D^{\alpha}_{|x|}} = \frac{\partial^\alpha }{\partial | x |^\alpha} \label{RieszOp}
\end{equation}
which is the governing operator of the symmetric process $S^{\alpha}_t$, $t>0$.  The Riemann-Liouville derivatives 
\begin{equation} 
\frac{d^\alpha f}{d (-x)^\alpha}(x) = (W^\alpha_{-} f)(x), \quad x \in \mathbb{R}_{+} \label{Rfracderr}
\end{equation}
and 
\begin{equation}
\frac{d^\alpha f}{d x^\alpha}= \frac{1}{\Gamma\left( n-\alpha \right)} \frac{d^n}{d x^n} \int_0^x (x-s)^{n-\alpha -1} f(s) \, ds, \quad x \in \mathbb{R}_{+} \label{Rfracder}
\end{equation}
are defined by restricting the Weyl's derivatives \eqref{Weyll1} and \eqref{Weyll2} to $(0, +\infty)$. For $\alpha \in \mathbb{N}$, the fractional derivatives above become ordinary derivatives  and 
\begin{equation}
\frac{d^\alpha}{d x^\alpha} = (-1)^\alpha \frac{d^\alpha}{d (-x)^\alpha}.\label{dernu1}
\end{equation}
We also introduce the Dzhrbashyan-Caputo fractional derivative 
\begin{equation}
 {^cD^{\alpha}_{t}} f(t) = \frac{1}{\Gamma\left( n-\alpha \right)} \int_0^t (t-s)^{n-\alpha -1} \frac{d^n f}{d s^n}(s) \, ds \label{Cfracder}
\end{equation}
defined for $n-1 < \alpha < n$, $n \in \mathbb{N}$, which is related to \eqref{Rfracder} as follows (see \citet{GM97} and \citet{KST06})
\begin{equation}
 {^cD^{\alpha}_{t}} f(t) = \frac{d^\alpha f}{d t^\alpha}(t) - \sum_{k=0}^{n-1} \frac{d^k f}{d t^k} (t) \Bigg|_{t=0^+} \frac{t^{k - \alpha}}{\Gamma(k - \alpha +1)}. \label{RCfracder}
\end{equation}

From the relation 
\begin{equation} 
Pr\{\mathfrak{L}^{\nu}_t <x\}= Pr\{\mathfrak{H}^{\nu}_x > t\}, \label{relHL}
\end{equation}
according to \cite{BM01, Dov2, MS08, BMN09ann}, we define the inverse process $\mathfrak{L}^{\nu}_t$, $t>0$ with law $l_{\nu}=l_{\nu}(x,t)$, $x,t>0$. As already mentioned before, $\mathfrak{H}^{\nu}_t$, $t>0$ is the $\nu$-stable subordinator, $\nu \in (0,1)$ with law, say, $h_{\nu}=h_{\nu}(x,t)$, $x,t>0$. The process $\mathfrak{H}^{\nu}_t$, $t>0$ is a process with non-negative, independent and homogeneous increments (see \citet{Btoi96}) whereas, the inverse process $\mathfrak{L}^{\nu}_t$, $t>0$ has non-negative, non-stationary and non-independent increments (see \citet{MSheff04}). Stable subordinators and their inverse processes are characterized by the following Laplace transforms:
\begin{equation}
E\exp\left( -\lambda \mathfrak{H}^{\nu}_t\right) = \exp\left( -t \lambda^{\nu}\right), \qquad E\exp\left( -\lambda \mathfrak{L}^{\nu}_t\right) = E_{\nu}(- \lambda t^\nu) \label{LapSpace}
\end{equation}
and
\begin{equation}
\mathcal{L}[h_{\nu}(x,\cdot)](\lambda) = x^{\nu-1}E_{\nu, \nu}(-\lambda x^{\nu}), \qquad \mathcal{L}[l_{\nu}(x, \cdot)](\lambda) = \lambda^{\nu-1}\exp\left( -x \lambda^{\nu} \right) \label{LapTime}
\end{equation}
where the entire function
\begin{equation}
E_{\alpha, \beta}(z) = \sum_{k \geq 0} \frac{z^k}{\Gamma(\alpha k +\beta)} , \quad z \in \mathbb{C}, \; \Re\{ \alpha \} >0, \; \beta \in \mathbb{C}  \label{GMLeffler}
\end{equation}
is the generalized Mittag-Leffler function for which
\begin{equation*}
\int_0^\infty e^{-\lambda z} z^{\beta -1} E_{\alpha, \beta}( - \mathfrak{c} z^\alpha) \, dz = \frac{\lambda^{\alpha -\beta}}{\lambda^{\alpha} + \mathfrak{c}}, \quad \Re\{ \lambda \} > |\mathfrak{c} |^{1/\alpha}, \; \Re\{\mathfrak{c} \} >0
\end{equation*}
and $E_{\alpha}(z) = E_{\alpha, 1}(z)$ is the Mittag-Leffler function. From \eqref{LapSpace} we immediately verify that the law $h_\nu$ satisfies the fractional equation $-\frac{\partial}{\partial t} h_{\nu}(x,t) = \frac{\partial^\nu}{\partial x^\nu} h_{\nu}(x, t)$ whereas, for the law of $\mathfrak{L}^{\nu}_t$, from \eqref{LapTime} we have that $\frac{\partial^\nu}{\partial t^\nu} l_{\nu}(x,t) = - \frac{\partial}{\partial x} l_{\nu}(x,t)$. Such density laws can not be represented in closed form. In this paper we write 
\begin{equation}
h_{\nu}(x,t) = \frac{1}{\nu t^{1/\nu}}H^{0,1}_{1,1}\left[ \frac{x}{t^{1/\nu}} \Bigg| \begin{array}{c} (1-\frac{1}{\nu}, \frac{1}{\nu})\\ (0,1) \end{array} \right], \;\; l_{\nu}(x,t) = \frac{1}{t^\nu}H^{1,0}_{1,1}\left[ \frac{x}{t^\nu} \Bigg| \begin{array}{c} (1-\nu, \nu)\\ (0,1) \end{array} \right] \label{hlHfox}
\end{equation}
for $x,t>0$, $\nu \in (0,1)$ in terms of H-functions as can be obtained by considering \eqref{mellinHfox} and the Mellin transforms (see \cite{Dov2, MLP01, MPG03})
\begin{align}
\mathcal{M}[h_{\nu}(\cdot, t)](\eta) = & \Gamma\left( \frac{1-\eta}{\nu} \right) \frac{t^\frac{\eta -1}{\nu}}{\nu \,\Gamma\left(1-\eta \right)}, \quad \mathcal{M}[l_{\nu}(\cdot, t)](\eta) = \frac{\Gamma\left(\eta \right)\, t^{\nu (\eta-1)}}{\Gamma\left( \eta \nu - \nu +1\right)}. \label{melTUTTE}
\end{align}
Further representations of $h_\nu$ and $l_\nu$ are given in terms of the Wright function 
\begin{equation*}
W_{\alpha, \beta}(z) = \sum_{k \geq 0} \frac{z^k}{k!\, \Gamma(\alpha k + \beta)}, \quad z \in \mathbb{C}, \; \Re\{ \alpha \} > -1,\; \beta \in \mathbb{C}
\end{equation*}
by considering that
\begin{equation*}
l_{\nu}(x, t) = \frac{1}{t^\nu} W_{-\nu, 1-\nu} \left(-\frac{x}{t^{\nu}} \right)
\end{equation*} 
and (see \citet{Dov4}) $xh_\nu(x, t) = tl_\nu(t, x)$.

\section{Fractional and Anomalous diffusions}
\label{sezFAD}
Standard diffusion has the mean squared displacement (or time-dependent variance) which is linear in time. Anomalous diffusion is usually met in disordered or fractal media (see e.g. \citet{GioRom92}) and represents a phenomenon for which the mean squared displacement is no longer linear but proportional to a power $\alpha$ of time with $\alpha \neq 1$. Thus we have superdiffusion ($\alpha>1$) or subdiffusion ($\alpha<1$) in which diffusion occurs faster or slower than normal diffusion (see e.g. \citet{Uch03}). 

\subsection{Fractional evolution equations}
We begin our analysis by studying anomalous diffusions on $\Omega_a=(0, a)$, $a>0$, whose governing equations involve the operator
\begin{align}
\mathcal{G}^{*} = & \frac{1}{\gamma^2} \frac{\partial}{\partial x} x^{\gamma \mu -\gamma + 1} \frac{\partial}{\partial x}  \frac{1}{\mathfrak{w}(x)} \label{OrdOpG}
\end{align}
($\mathfrak{w} (x) = x^{\gamma \mu -1}$ will play the role of weight function further on) for $\gamma=\pm 1$ and $\mu>0$.

The solutions to the ordinary problem
\[ \frac{\partial u}{\partial t} = \mathcal{G}^{*}\, u \]
subject to the initial data $u_0=f$ can be written as $u(x,t) = T_1(t)f(x)$ where $T_1(t) = \exp -t \mathcal{G}^{*}$. For the subordination principle (see for example \citet{Bazh00, Boch49}) we can write the solutions to the fractional problem
\[ \frac{\partial^\nu u}{\partial t^\nu}= \mathcal{G}^{*}\, u, \quad \nu \in (0,1] \]
($\frac{\partial^\nu u}{\partial t^\nu}$ is defined in \eqref{Rfracder}) subject to the initial condition $u_0=f$ by considering the convolution operator
\[ T_\nu(t) = \frac{1}{t^{\nu}} \int_0^\infty ds\,  H^{1,0}_{1,1}\left[ \frac{s}{t^\nu} \Bigg| \begin{array}{c} (1-\nu, \nu)\\ (0,1) \end{array} \right]  e^{-s \mathcal{G}^{*}} \]
where the Fox's function $H^{1,0}_{1,1}$ is the law of the inverse process $\mathfrak{L}^\nu_t$ introduced in the previous section. In order to explicitly write the convolution $T_\nu(t)f(x)$, for $\gamma \neq 0$, $\mu >0$, we introduce the functions
\begin{equation}
g^{\gamma}_{\mu}(x,t) = \textrm{sign}(\gamma) \frac{1}{t}Q^{\gamma}_{\mu}\left( \frac{x}{t} \right) \quad \textrm{and} \quad \tilde{g}^{\gamma}_{\mu}(x,t) = g^{\gamma}_{\mu}(x, t^{1/\gamma}) \label{differentg}
\end{equation}
where
\begin{equation*}
Q^{\gamma}_{\mu}(z) = \gamma \frac{z^{\gamma \mu -1}}{\Gamma(\mu)} e^{- z^{\gamma}}, \quad z>0,\,\gamma > 0,\, \mu >0
\end{equation*}
is the well known generalized gamma density or Weibull distribution if $\gamma \in \mathbb{N}$.  

\begin{te}
For $\gamma \neq 0$, $\mu >0$, $\nu \in (0,1]$, the solution to the fractional p.d.e.
\begin{equation}
\left\lbrace \begin{array}{ll}
\frac{\partial^\nu \tilde{u}^{\gamma, \mu}_{\nu}}{\partial t^\nu}  = \mathcal{G}^{*}\, \tilde{u}^{\gamma, \mu}_{\nu}, & x \in \Omega_{\infty}, \, t>0\\ 
\tilde{u}^{\gamma, \mu}_{\nu}(x,0)=\delta(x)
\end{array} \right .
\label{tmpDG}
\end{equation}
is given by
\begin{equation}
\tilde{u}^{\gamma, \mu}_{\nu} (x,t) = \int_0^\infty \tilde{g}^{\gamma}_{\mu}(x, s)\, l_{\nu}(s,t) \, ds \label{tmpDGlaw}
\end{equation}
where $\tilde{g}^\gamma_\mu$ and $l_\nu$ are defined in \eqref{differentg} and \eqref{hlHfox} respectively. 
\label{dfjlY}
\end{te}
From the fact that 
$$\lim_{\nu \to 1} \mathfrak{L}^\nu_t \stackrel{a.s}{=} t$$ 
where $t$ is the elementary subordinator (see \citet{Btoi96}) we can write
\[ \lim_{\nu \to 1} l_\nu(s,t) = \delta(s-t) \]
and thus $\tilde{u}^{\gamma, \mu}_{1} = \tilde{g}^\gamma_\mu$ are the solutions to \eqref{tmpDG} for $\nu=1$ and $\gamma \neq 0$, $\mu>0$. In this case the time-fractional derivative becomes the ordinary derivative $\partial /\partial t$. For $\gamma=\pm 1$ and $\mu >0$, we obtain that
\[ \tilde{g}^{1}_\mu(x,t) = E^x \delta(G_\mu(t)), \]
\[ \tilde{g}^{-1}_\mu(x,t) = E^x \delta(E_\mu(t)),\]
\[ \tilde{u}^{1,\mu}_{\nu}(x,t) = E^x \delta(G_\mu(L^\nu_t)), \]
\[ \tilde{u}^{-1,\mu}_{\nu}(x,t) = E^x \delta(E_\mu(L^\nu_t)) \]
where $E^x \delta(X_t) = \delta * f_{X_t}(x)$. The process $G_\mu$ is a non-negative diffusion satisfying the stochastic differential equation 
\begin{equation}
dG_\mu(t) = \mu \, dt + 2 \sqrt{G_\mu(t)}\, dB_1(t) \label{SDEBes}
\end{equation}
where $B_1(t)$, $t>0$ is a Brownian motion with variance $t/2$. The reciprocal gamma law $e_{\mu}=g^{-1}_{\mu}$, $\mu>0$, represents the $1$-dimensional marginal law of the process satisfying the stochastic equation
\begin{equation}
dE_{\mu}(t)= - \left( E_{ \mu}(t) - \frac{1}{\mu -1} \right) \,dt + \sqrt{\frac{2\, |E_{\mu}(t)|^2}{\mu -1}}\, dB_2(t) \label{SPDEdiE}
\end{equation}
where $B_2(t)$, $t>0$ is a standard Brownian motion (see e.g. \citet{BSS05}, \citet{Pes06}). Due to the global Lipschitz condition on both coefficients, the stochastic equation \eqref{SPDEdiE} has a unique solution which is a strong Markov process. The process $E_\mu$ also appears by considering the integral of a geometric Brownian motion with drift $\mu$, that is 
\begin{equation}
\frac{1}{2} E_\mu \stackrel{law}{=} \int_0^\infty \exp\left( 2B(s) - 2\mu s\right) ds
\end{equation}
see \citet{Dufresne90, PolSieg85}. For $\gamma=2$, the operator $\mathcal{G}^{*}$ becomes the governing operator of a $2\mu$-dimensional Bessel process $R_{2\mu} = G_{2\mu}^{1/2}$. 

For the processes $G_\mu$ and $E_\mu$ introduced above there exist a couple of interesting properties. In particular, we have that
\begin{equation*}
Pr\{ G_\mu(x) > t\} = Pr\{ E_\mu(t) < x\}
\end{equation*}
and therefore we refer to $E_\mu$ as the inverse to $G_\mu$, whereas from the fact that
\begin{equation*}
E_\mu(t) \stackrel{law}{=} 1/G_\mu(t)
\end{equation*}
the process $E_\mu$ will be also termed reciprocal process of $G_\mu$.

\begin{os}
\normalfont
We observe that $\mathfrak{H}^\nu_t$ has non-negative increments and therefore, from \eqref{relHL}, the inverse process $\mathfrak{L}^\nu_t$ can be regarded as an Hitting time. This does not hold for $E_\mu$ being $G_\mu$ a  diffusion driven by \eqref{SDEBes}.
\end{os}

The operator $\mathcal{G}^{*}$ is the adjoint of the infinitesimal generator
\begin{align}
\mathcal{G} = & \frac{1}{\gamma^2\, \mathfrak{w}(x)} \frac{\partial}{\partial x} x^{\gamma\mu -\gamma +1} \frac{\partial}{\partial x} , \quad \gamma=\pm 1, \, \mu >0\label{opGSL}
\end{align}
which is a second order differential operator driving the squared Bessel process $G_\mu$ if $\gamma=+1$ and the inverse process $E_\mu$ if $\gamma=-1$. The Sturm-Liouville eigenvalue problem (see \citet{Boc29, CouHilb}) associated with \eqref{opGSL} leads to the differential equation  (see Lemma \ref{lmG} below)
\begin{equation}
\mathcal{G} \, \bar{\psi}_{\kappa_i} = - (\kappa_i /2)^2 \, \bar{\psi}_{\kappa_i}.
\end{equation}
The eigenfunctions
\begin{equation}
\bar{\psi}_{\kappa_i}(x) = x^{\frac{\gamma}{2}(1-\mu)} J_{\mu -1}\left( \kappa_i \, x^{\gamma/2}\right) \label{eigFB}
\end{equation}
corresponding to different eigenvalues are orthogonal with respect to the weight function $\mathfrak{w}(x) = x^{\gamma \mu -1}$ in the sense that
\begin{equation}
\int_{\Omega_a}  \bar{\psi}_{\kappa_i}\left(\frac{x}{a}\right)  \bar{\psi}_{\kappa_j}\left(\frac{x}{a}\right)  \mathfrak{w}(x) dx = 0, \quad \textrm{if }\; i \neq j \label{orthoSL}
\end{equation}
on the domain $\Omega_{a}=(0, a)$, $a>0$. The eigenvalues are written in terms of $\{\kappa_i \}_{i \in \mathbb{N}}$ which are the zeros of the Bessel function of the first kind (see \cite[p. 102]{LE})
\begin{equation}
J_{\alpha}(z) = \sum_{k \geq 0} \frac{(-1)^k \, (z/2)^{\alpha + 2k}}{k!\, \Gamma(\alpha + k+1)}, \quad |z| < \infty, \quad |\textrm{arg} z| < \pi .\label{BesJ}
\end{equation}
The set of eigenfunctions $\{\bar{\psi}_{\kappa_i}\}_{i \in \mathbb{N}}$ is complete and therefore a piecewise smooth function can be represented by a generalized Fourier series expansion of \eqref{eigFB}. In particular, on the finite domain $\Omega_{1} = (0,1)$, we study the solution to
\begin{equation}
\left\lbrace \begin{array}{ll} 
{^cD^\nu_t} m^{\gamma, \mu}_{\nu} = \mathcal{G}^{*}\, m^{\gamma, \mu}_{\nu}, & x \in \Omega_{1},\, t>0,\\ m^{\gamma, \mu}_{\nu}(x,0)=m_0(x), & m_0 \in C(\Omega_{1})\\ 
m^{\gamma,\mu}_{\nu}(x,t)=0, & x \in \partial \Omega_{1}, \; t>0,
\end{array} \right .
\label{probBoundG} 
\end{equation}
with $\nu \in (0,1]$, $\gamma \neq 0$, $\mu >0$ and arrive at the next result. 
\begin{te}
The solution to the problem \eqref{probBoundG} can be written as follows
\begin{equation}
m^{\gamma, \mu}_{\nu}(x,t) = \mathfrak{w}(x) \, \sum_{n=1}^{\infty} c_n \, E_{\nu}\left(- (\kappa_n /2)^2\, t^{\nu}\right) \, \frac{\bar{\psi}_{\kappa_n}(x)}{\| \bar{\psi}_{\kappa_n} \|^2_{\mathfrak{w}}} 
\end{equation}
where $\mathfrak{w}(x)=x^{\gamma \mu -1}$ is the weight function,
\begin{equation}
c_n = \int_{\Omega_1} m_0(x) \bar{\psi}_{\kappa_n}(x) \, dx, \quad n =1,2, \ldots 
\end{equation}
and $E_{\nu}(z) = E_{\nu, 1}(z)$ is the Mittag-Leffler function \eqref{GMLeffler}.
\end{te}
This result can be easily extended to the case in which the bounded domain is $\Omega_{a}=(0, a)$ for $a>0$ and we can write
\begin{equation}
m^{1,\mu}_{\nu}(x,t) = E^{x}\, m_0(G_\mu(\mathfrak{L}^{\nu}_t))\, \mathbf{1}_{(\mathfrak{L}^{\nu}_t < T_{\Omega_{a}}(G_\mu))}
\end{equation}
and
\begin{equation}
m^{-1,\mu}_{\nu}(x,t) = E^{x}\, m_0(E_\mu(\mathfrak{L}^{\nu}_t))\, \mathbf{1}_{(\mathfrak{L}^{\nu}_t < T_{\Omega_{a}}(E_\mu))}
\end{equation}
where $T_{D}(X) = \inf\{t \geq 0:\, X_t \notin D \}$ is an exit time and $E^x \phi(X_t) = \phi * f_{X_t}(x)$. We observe that $\mathbf{1}_{(L^{\nu}_{t} < T_{\Omega_{a}}(X))}=\mathbf{1}_{(t < T_{\Omega_{a}}(X(L^{\nu})))}$. For $a \to \infty$ we obtain the solutions to the problem \eqref{tmpDG}. Indeed, we have that $\mathbf{1}_{(L^{\nu}_{t} < T_{\Omega_{\infty}}(G_{\mu}))} \equiv 1$ and $\mathbf{1}_{(L^{\nu}_{t} < T_{\Omega_{\infty}}(E_{\mu}))} \equiv 1$. Furthermore, formula \eqref{orthoSL} can be rewritten for $a \to \infty$ as
\begin{equation*}
\int_0^\infty \bar{\psi}_{\kappa_i}(x)  \bar{\psi}_{\kappa_j}(x)  \mathfrak{w}(x) dx = \delta(\kappa_i - \kappa_j) / \kappa_j
\end{equation*}
which leads to the Hankel transforms 
\begin{equation*}
(\mathcal{H}\,f)(\rho) = \int_0^\infty x J_{\nu}(\rho x) f(x) dx \quad \textrm{and} \quad f(x) = \int_0^\infty \rho J_{\nu}(\rho x) (\mathcal{H}\,f)(\rho) d\rho
\end{equation*} 
of a well-defined function $f$.

\subsection{Fractional powers of operators}
From the Cauchy integral
\[ \frac{1}{2\pi i} \int_{\Gamma} \frac{f(z)\, dz}{(z-z_0)} \]
we can define an algebraic isomorphism such that a function of a linear bounded operator $A$ is defined as
\[ f(A) = \frac{1}{2\pi i} \int_{\Gamma} f(z)\, R(z, A)\,dz  \]
where $R(z,A) = (zI - A)^{-1}$ is the resolvent operator (under conditions \eqref{cvbA}). In general, for a closed linear operator $A$ in a Banach space $X$, the definition of $A^\alpha$ for a complex $\alpha$ could be given by means of the Dunford integral
\begin{align*} 
A^\alpha = & \frac{1}{2\pi i} \int_{\Gamma} \zeta^\alpha \, (\zeta - A)^{-1}\, d\zeta = - \frac{\sin \pi \alpha}{\pi} \int_0^\infty \lambda^\alpha \, (\lambda + A)^{-1}\, d\lambda
\end{align*}
where $\Gamma$ encircles the spectrum $\sigma(A)$ counterclockwise avoiding the negative real axis and $\zeta^\alpha$ takes the principal branch (see e.g. \cite{Balak60, HoWe72, Kom66, KraSob59, Wata61} and the references therein). For such operators the expected property $A^\alpha A^\beta = A^{\alpha + \beta}$ holds true. A well-known example is the fractional Laplace operator which can be also defined (in the space of Fourier transforms) as
\begin{equation}
\triangle^{\alpha/2}u(\mathbf{x}) = -\frac{1}{2\pi} \int_{\mathbb{R}^n} e^{-i \boldsymbol{\xi} \cdot \mathbf{x}} \| \boldsymbol{\xi} \|^{\alpha} \, \mathcal{F}[u](\boldsymbol{\xi})\, d\boldsymbol{\xi}, \quad \mathbf{x} \in \mathbb{R}^n. \label{opLaplFrac}
\end{equation}
The stochastic solution to the Cauchy problem involving the fractional Laplace operator \eqref{opLaplFrac} is given by the process $\mathbf{B}(\mathfrak{H}^{\alpha}_t)$ ($\mathbf{B}$ is a Brownian motion driven by the self-adjoint Laplace operator $\triangle$ and $\mathfrak{H}^{\alpha}_t$ is a stable subordinator) and has been first investigated by  \citet{Boch49, Feller52}. In the one-dimensional case, the operator \eqref{opLaplFrac} becomes the Riesz operator \eqref{RieszOp} for which the representation \eqref{weylrap} given by means of the right and left Weyl's derivatives holds as well. 

In this section we study a fractional power of \eqref{OrdOpG} which is the power of the composition of non commuting operators and obtain the explicit representation
\begin{equation} 
\mathcal{A}\, f = -(-\mathcal{G}^{*})^{\nu}\,f = - \frac{\partial^\nu}{\partial x^\nu} \bigg( x^{\mu - 1 +\nu} \frac{\partial^\nu}{\partial (-x)^\nu}\Big( x^{1-\mu}\, f \Big) \bigg) \label{FfracOper}
\end{equation}
defined on the positive real line $\Omega_{\infty}=(0, +\infty)$ for $\nu \in (0,1)$. This representation involves the Riemann-Liouville fractional derivatives \eqref{Rfracderr} and \eqref{Rfracder} which replace the Weyl's derivatives \eqref{Weyll1} and \eqref{Weyll2} for functions defined on the positive real line. As we can check, from \eqref{dernu1}, the fractional operator \eqref{FfracOper} for $\nu=1$ becomes the operator \eqref{OrdOpG}. We show that, for $\nu \in (0,1)$, the process $G_\mu(\mathfrak{H}^\nu_t)$ is the stochastic solution to the Cauchy problem involving the operator \eqref{FfracOper} where $G_\mu$ is driven by $\mathcal{G}^{*}$. First we state the following result.
\begin{te}
Let us consider the process $G_{\mu}(t)$, $t>0$, satisfying the stochastic equation \eqref{SDEBes}. For $\mu >0$, $\forall \, \nu \in (0,1]$, the $1-$dimensional density law $g^1_\mu$ of the process $G_\mu$ solves the fractional p.d.e.\ on $\Omega_{\infty}=(0, +\infty)$
\begin{equation}
\frac{\partial^\nu g^{1}_{\mu}}{\partial t^\nu}(x,t) =  \frac{\partial^\nu}{\partial (-x)^\nu} \bigg( x^{\mu - 1 +\nu} \frac{\partial^\nu}{\partial (-x)^\nu} \Big( x^{1 - \mu}\, g^{1}_{\mu}(x,t) \Big) \bigg) \label{Mguno}
\end{equation}
subject to the initial condition $g^1_\mu(x, 0)=\delta(x)$.
\end{te}
The stochastic solution to \eqref{Mguno} is the process $G_\mu$ which does not depend on the fractional index $\nu \in (0,1]$. Furthermore, we notice that the space operator appearing in \eqref{Mguno} is different from $\mathcal{A}$.\\

We proceed our analysis by considering the process $F^{\nu, \beta}_t = \mathfrak{H}^{\nu}_{\mathfrak{L}^{\beta}_t}$, $t>0$ with law
\begin{equation}
\mathfrak{f}_{\nu, \beta}(x,t) = \int_0^\infty h_{\nu}(x, s) \, l_{\beta}(s, t) \, ds , \quad x \geq 0,\; t>0,\; \nu, \beta \in (0,1) \label{lawf}
\end{equation}
which has been thoroughly studied by several authors, see e.g.\ \cite{ Dov2, DO2, Lanc, MLP01, MPS05}. If $\nu=\beta$, then the law \eqref{lawf} takes the form $\mathfrak{f}_{\nu, \nu}(x,t)=t^{-1}f_{\nu}(t^{-1}x)$ where
\begin{equation}
f_{\nu}(x) = \frac{1}{\pi} \frac{x^{\nu -1} \sin \pi \nu}{1+ 2 x^{\nu}\cos \pi \nu + x^{2\nu}}, \quad x\geq 0,\; t>0, \quad \nu \in (0,1) \label{dratio}
\end{equation}
and $F^{\nu, \nu}_t \stackrel{law}{=} t \times \, _1 \mathfrak{H}^{\nu}_t / \, _2\mathfrak{H}^{\nu}_t$, $t>0$, (which means that $F^{\nu, \nu}_t  \in \mathbb{P}_{1}$) where $\, _j\mathfrak{H}^{\nu}_t$, $j=1,2$ are independent stable subordinators and the ratio $\, _1 \mathfrak{H}^{\nu}_t / \, _2\mathfrak{H}^{\nu}_t$ is independent of $t$ (see \cite{ChYor03, Dov2, Lamp63, Zol57}). The density law $f_\nu$ arises in many important contexts, we refer to the paper by \citet{Lanc} and the references therein for details. 

\begin{lem}
The governing equation of the process $F^{\alpha, \beta}_t$, $t>0$, with density law \eqref{lawf}, is written as
\begin{equation}
\left( \frac{\partial^\beta}{\partial t^\beta} + \frac{\partial^\nu}{\partial x^\nu} \right) \mathfrak{f}_{\nu, \beta}(x,t)=\delta(x)\frac{t^{-\beta}}{\Gamma(1-\beta)}, \quad x \geq 0,\; t>0 \label{eq1lamp}
\end{equation}
with $\mathfrak{f}_{\nu,\beta}(\partial \Omega_{\infty},t) =0$ and $\mathfrak{f}_{\nu, \beta}(x,0)=\delta(x)$ or, by considering \eqref{RCfracder},
\begin{equation}
\left( {^cD^\beta_{t}} + \frac{\partial^\nu}{\partial x^\nu} \right)\mathfrak{f}_{\nu, \beta}(x,t)= 0, \quad x \geq 0,\; t>0 \label{eq2lamp}
\end{equation}
with $\mathfrak{f}_{\nu, \beta}(x,0)=\delta(x)$.
\end{lem}
\begin{proof}
From the Laplace transforms \eqref{LapSpace} and \eqref{LapTime} we have that
\begin{equation*}
\Psi(\xi, \lambda) = \int_0^\infty e^{-\lambda t} \left( E\, e^{-\xi F^{\alpha, \beta}_t} \right) dt = \lambda^{\beta -1}(\lambda^\beta + \xi^\nu)^{-1}.
\end{equation*}
Let us consider the equation \eqref{eq2lamp}. From the fact that 
\begin{equation*}
\mathcal{L}\left[ {^cD^{\beta}_{t}} f \right] (\lambda) = \lambda^\beta \mathcal{L}[f](\lambda) - \lambda^{\beta -1} f(0^{+}) , \quad \beta \in (0,1)
\end{equation*}
(which comes from \eqref{RCfracder}) we obtain
\begin{equation*}
\lambda^{\beta} \Psi(\xi, \lambda) - \lambda^{\beta -1} + \xi^{\nu}\Psi(\xi, \lambda) = 0
\end{equation*}
which concludes the proof.
\end{proof}

We state the main result of this section concerning the operator \eqref{FfracOper}.
\begin{te}
For $x \in \Omega_{\infty}=(0, +\infty)$, $t>0$, $\mu >0$ and $\beta, \nu \in (0,1]$ we have that
\begin{itemize}
\item [i)] the density law $\mathfrak{g}^{\nu}_{\mu}(x,t) = \int_0^\infty g^{1}_{\mu}(x, s)\,  h_{\nu}(s,t) \,ds$ solves the fractional p.d.e.
\begin{equation}
\frac{\partial \mathfrak{g}^{\nu}_{\mu}}{\partial t} (x,t) = \mathcal{A}\, \mathfrak{g}^{\nu}_{\mu}(x,t), \label{Mgdue}
\end{equation}
\item [ii)] the density law $\mathbf{g}^{\nu, \beta}_{\mu}(x,t) = \int_0^\infty g^{1}_{\mu}(x, s)\, \mathfrak{f}_{\nu, \beta}(s ,t) \,ds$ solves the fractional p.d.e.
\begin{equation}
\frac{\partial^\beta \mathbf{g}^{\nu, \beta}_{\mu}}{\partial t^\beta}(x,t) = \mathcal{A}\, \mathbf{g}^{\nu, \beta}_{\mu}(x,t). \label{Mgtre}
\end{equation}
\end{itemize}
Furthermore, 
\begin{equation}
\mathbf{g}^{\nu, \beta}_{\mu}(x,t)  = \frac{1}{t^{\beta/\nu}} \mathbf{G}^{\nu, \beta}_{\mu}\left( \frac{x}{t^{\beta/\nu}} \right)
\end{equation}
where
\begin{equation*}
\mathbf{G}^{\nu, \beta}_{\mu}(x) = \frac{1}{x} H^{2,1}_{3,3} \left[ x \bigg| \begin{array}{ccc} (1, \frac{1}{\nu}); & (1, \frac{\beta}{\nu}); & (\mu,0) \\ (\mu , 1); & (1, \frac{1}{\nu}) ;& (1,1) \end{array} \right], \quad x>0
\end{equation*}
is a Fox's function defined in \eqref{defHFoxApp}.
\label{Teogbold}
\end{te}
From Theorem \ref{Teogbold} we have that 
\[ \mathfrak{g}^{\nu}_{\mu}(x,t) = E^x\delta(G_\mu(\mathfrak{H}^{\nu}_t)), \]
\[ \mathbf{g}^{\nu, \beta}_{\mu}(x,t) = E^x\delta(G_\mu(F^{\nu, \beta}_t)) \]
and thus  $G_\mu(F^{\nu, \beta}_t)$, $t>0$ is the stochastic solution to \eqref{Mgtre} whereas $G_\mu(F^{\nu, 1}_t) = G_\mu(\mathfrak{H}^\alpha_t)$, $t>0$ represents the stochastic solution to \eqref{Mgdue}. This is because of the fact that $\mathfrak{H}^{\nu}_{\mathfrak{L}^{1}_t} \stackrel{a.s.}{=} \mathfrak{H}^{\nu}_t$, $t>0$, being $\mathfrak{L}^{1}_t \stackrel{a.s.}{=} t$ the elementary subordinator and $l_\nu(s,t) \to \delta(t-s)$ for $\nu \to 1$. Furthermore, from the formulas \eqref{mellinHfox} and \eqref{intStrip}, by taking into account \eqref{MelGProof}, we immediately have that
\begin{equation*}
\mathbf{G}^{\nu, \nu}_{\mu}(x) = \frac{1}{x} H^{1,1}_{2,2} \left[ x \bigg| \begin{array}{ccc} (1, \frac{1}{\nu});  & (\mu,0) \\ (\mu , 1); & (1, \frac{1}{\nu})  \end{array} \right], \quad x>0, \; \nu \in (0,1),
\end{equation*}
\begin{equation*}
\mathbf{G}^{\nu, 1}_{\mu}(x) = \frac{1}{x} H^{1,1}_{2,2} \left[ x \bigg| \begin{array}{cc} (1, \frac{1}{\nu}); & (\mu,0) \\ (\mu , 1) ;& (1,1) \end{array} \right], \quad x>0, \; \nu \in (0,1)
\end{equation*}
and
\begin{equation*}
\mathbf{G}^{1, 1}_{\mu}(x) = \frac{1}{x} H^{1,0}_{1,1} \left[ x \bigg| \begin{array}{c}  (\mu,0) \\ (\mu , 1) \end{array} \right], \quad x>0.
\end{equation*}
For $\beta=\nu$, the equation \eqref{Mgtre} takes the form
\begin{equation}
\frac{\partial^\nu \mathbf{g}^{\nu, \nu}_{\mu}}{\partial t^\nu} (x,t) =  - \frac{\partial^\nu}{\partial x^\nu} \bigg( x^{\mu - 1 +\nu} \frac{\partial^\nu}{\partial (-x)^\nu} \Big( x^{1 - \mu}\, \mathbf{g}^{\nu, \nu}_{\mu}(x,t) \Big) \bigg) \label{pdeLSLSL}
\end{equation}
which differs form \eqref{Mguno} and represents the governing equation of the process 
\[ G_\mu\left( t \times \, {_1 \mathfrak{H}}^{\nu}_1 / \, {_2\mathfrak{H}^{\nu}}_1 \right), \quad t>0 \]
where $\, _j\mathfrak{H}^{\nu}_t$, $t>0$, $j=1,2$ are independent stable subordinators. Here, the stochastic solution to \eqref{pdeLSLSL} depends on $\nu$ only by means of the ratio ${_1 \mathfrak{H}}^{\nu}_1 / \, {_2\mathfrak{H}}^{\nu}_1$ which possesses distribution \eqref{dratio}. For $\beta=1$, equation \eqref{Mgtre} becomes \eqref{Mgdue} and writes
\begin{equation}
\frac{\partial \mathbf{g}^{\nu, 1}_{\mu}}{\partial t}(x,t) =  - \frac{\partial^\nu}{\partial x^\nu} \bigg( x^{\mu - 1 +\nu} \frac{\partial^\nu}{\partial (-x)^\nu} \Big( x^{1 - \mu}\, \mathbf{g}^{\nu, 1}_{\mu}(x,t) \Big) \bigg) \label{pdeLSLSLuno}
\end{equation}
($\mathfrak{g}^{\nu}_{\mu} = \mathbf{g}^{\nu, 1}_{\mu}$) with stochastic solution given by $G_\mu(\mathfrak{H}^\nu_t), \quad t>0$. Finally, for $\nu=1$ in \eqref{Mgtre}, we reobtain the governing equation of the process $G_\mu(L^\nu_t)$, $t>0$ already investigated in Theorem \ref{dfjlY}. 

\begin{os}
\normalfont
We observe that, for $\mu \in \mathbb{N}$,
\begin{equation*}
G_\mu(\mathfrak{H}^\nu_t) = \| \mathbf{B}(\mathfrak{H}^\nu_t) \|^2 =  \sum_{j=1}^{\mu} \left[ {_j B(\mathfrak{H}^\nu_t)} \right]^{2}
\end{equation*}
where $\mathbf{B}(t) = ({_1B(t)}, \ldots, {_nB(t)})$, $t>0$  and ${_jB(t)}$, $j=1, \ldots , n$ are independent Brownian motions. From the Bochner's subordination rule we get ${_jB(\mathfrak{H}^\nu_t)}\stackrel{law}{=}{_jS^{2\nu}_t}$ which are symmetric stable processes with $E\exp i\xi {_jS^{2\nu}_t} = \exp -t |\xi |^{2\nu}$ for all $j=1,\ldots , n$. Thus, by considering $n$ independent stable processes $_jS^{2\nu}_t$ and $\mathbf{S}^{2\nu}_t = ({_1S^{2\nu}_t}, \ldots , {_nS^{2\nu}_t})$, $t>0$, we obtain that $G_\mu(\mathfrak{H}^\nu_t) \stackrel{law}{=} \| \mathbf{S}^{2\nu}_t \|^2$ . 
\end{os}

\begin{os}
\normalfont
For $\mathbf{B}(t) \in \mathbb{R}^3$, the three dimensional Bessel process represents a radial diffusion on a homogeneous ball. The subordinated squared Bessel process can be therefore regarded as a radial diffusion on a non-homogeneous ball, with fractal structure for instance. This interpretation is due to the fact that the random time $\mathfrak{H}^\nu_t$ has non-negative increments and therefore, non-decreasing paths. Furthermore, the subordinated process $G_\mu(\mathfrak{H}^\nu_t)$ speed up as $\mathfrak{H}^\nu_t$ increases. For $\nu\to1$, $\mathfrak{H}^\nu_t \to t $ a.s. and $G_\mu(\mathfrak{H}^\nu_t)$ becomes standard diffusion because of the linear growing of time.
\end{os}

\begin{os}
\normalfont
From \eqref{meDAtmp} we obtain that
\begin{equation*}
E\left[ G_\mu(F^{\nu, \beta}_t) \right]^r \propto t^{\frac{\beta}{\nu}r}, \quad r>0, \; \nu, \beta \in (0, 1] .
\end{equation*} 
Moreover, $F^{\nu, \beta}_t \to \mathfrak{L}^\beta_t$ for $\nu \to 1$ and therefore we have that $G_\mu(\mathfrak{L}^\beta_t)$ is a subdiffusion whereas, from the fact that $F^{\nu, \beta}_t \to \mathfrak{H}^\nu_t$ for $\beta \to 1$ we get the superdiffusion $G_\mu(\mathfrak{H}^\nu_t)$.
\end{os}

\subsection{Explicit representations of solutions via Mellin convolutions}
The laws of the processes $\mathfrak{H}^{\nu}_t$ and $\mathfrak{L}^{\nu}_t$ can be written in terms of H-functions as pointed out in Section \ref{intNot}. Alternative expressions can be given in terms of the Wright function but only for $\nu=1/2, 1/3$ we obtain a closed form of the density laws and therefore of the solutions investigated in the previous section. 

We give an explicit representation of the solutions presented so far by exploiting the Mellin convolutions of generalized gamma functions. The simplest convolutions we deal with are written below: for $x ,t>0$ and $\gamma \neq 0$, $\mu_1$, $\mu_2>0$,
\begin{equation}
g^{\gamma}_{\mu_1} \star g^{- \gamma}_{\mu_2}(x,t) = \frac{|\gamma |}{B(\mu_1, \mu_2)} \frac{x^{\gamma \mu_1 -1} t^{\gamma \mu_2}}{(t^\gamma + x^\gamma)^{\mu_1 + \mu_2}} \label{gConv}
\end{equation}
where $B(\cdot, \cdot)$ is the Beta function (see e.g. \citet[formula 8.384]{GR}) and,
\begin{equation}
g^{\gamma}_{\mu_1} \star g^{\gamma}_{\mu_2} (x,t) = \frac{2| \gamma | \, \left(x^\gamma /t^\gamma \right)^\frac{\mu_1+\mu_2}{2}}{x\, \Gamma(\mu_1)\Gamma(\mu_2)} K_{\mu_2-\mu_1} \left( 2 \sqrt{\left(x^\gamma / t^\gamma \right)} \right),  \label{eConv}
\end{equation}
where $K_{\alpha}$ is the modified Bessel function of the second kind. In particular
\begin{align}
K_{\alpha} (z) = \frac{\pi}{2} \frac{I_{-\alpha}(z) - I_{\alpha}(z)}{\sin \alpha \pi}, \quad \alpha \textrm{ not integer} \label{functionK}
\end{align}
(see \cite[formula 8.485]{GR}) where
\begin{align}
I_{\alpha}(z) = \sum_{k \geq 0} \frac{(z/2)^{\alpha + 2k}}{k!\, \Gamma(\alpha + k+1)}, \quad |z| < \infty, \quad |\textrm{arg} z| < \pi
\end{align}
is the modified Bessel function of the first kind (see \cite[formula 8.445]{GR}).
\begin{df} 
For $-\infty < a < b < \infty$, we define the space
\begin{equation*}
\mathbb{M}_{a}^{b} = \left\lbrace  f: \mathbb{R}_{+} \mapsto \mathbb{C} \, \big|  \, x^{\eta-1}f(x) \in L^{1}(\mathbb{R}_{+}), \forall \eta \in \mathbb{H}^{b}_{a} \right\rbrace
\end{equation*}
where $\mathbb{H}^{b}_{a} = \{ \zeta:\, \zeta \in \mathbb{C}, a < \Re\{\zeta \} < b  \}$.
\label{defM}
\end{df}
\begin{df} 
Let us consider the function $\mathcal{P}_Y: \mathbb{C} \mapsto \mathbb{R}$. We define the class of one-dimensional processes 
\begin{equation*} 
\mathbb{P}_{\alpha} = \left\lbrace Y(t), \, t>0 \,:\, \exists\, \mathbb{S} \subset \mathbb{H}_{a}^{b} \; s.t. \; E \big| Y(t^\alpha) / t \big|^{\eta -1}=\mathcal{P}_Y(\eta), \; \forall \, \eta \in \mathbb{S}  \right\rbrace, \quad \alpha \in \mathbb{R}.
\end{equation*}
\end{df}
\begin{os}
\normalfont
We notice that $S^{\nu}_t \in \mathbb{P}_{\nu} \; \Leftrightarrow \; S^{\nu}_t \stackrel{law}{=} t^{1/\nu}\, S^{\nu}_1$.
\end{os}
\begin{df}
We define the class of functions 
$$\mathbb{F}_{\alpha} = \left\lbrace f \, \big| \, Y \sim f, \, Y \in \mathbb{P}_{\alpha} \right\rbrace$$
where $Y \sim f$ means that the process $Y$ possesses the density law $f$. 
\end{df}
\begin{os}
\normalfont
We remark that $\mathbb{F}_{\alpha} \subset \mathbb{M}_{a}^{b}$.
\end{os}
We point out that for a composition involving the processes $Y_{\sigma_j} = \left[ X_{\sigma_j}\right]^\alpha$ where $X_{\sigma_j}$ are independent processes such that $X_{\sigma_j} \in \mathbb{P}_{\alpha}$ for all $j=1,2,\ldots n$, we have that $Y_{\sigma_j} \in \mathbb{P}_{1}$, $\forall j$ which implies that
\begin{equation}
Y_{\sigma_1}(Y_{\sigma_2}(\ldots Y_{\sigma_n}(t) \ldots)) \stackrel{law}{=}  Y_{\sigma_1}(t^{1/n}) Y_{\sigma_2}(t^{1/n}) \cdot \cdot \cdot Y_{\sigma_n}(t^{1/n})   \label{ProdVarM}
\end{equation}
for all possible permutations of $\{\sigma_j \}$, $j=1,2,\ldots , n$. This can be easily carried out by observing that $\mathcal{P}_{Y_{\sigma_j}}(\eta) = \mathcal{P}_{X_{\sigma_j}}(\eta\alpha - \alpha +1)$. Indeed, 
\begin{equation*}
E [ X_{\sigma_j}(t) ]^{\eta -1} = \mathcal{P}_{X_{\sigma_j}}(\eta) \, t^\frac{\eta -1}{\alpha}\quad \Rightarrow \quad E[Y_{\sigma_j}(t)]^{\eta -1} = E[X_{\sigma_j}]^{(\eta \alpha - \alpha +1)-1 }
\end{equation*}
and therefore, for $i \neq j$,
\begin{equation*}
E[Y_{\sigma_j}(Y_{\sigma_i}(t))]^{\eta -1} = \mathcal{P}_{X_{\sigma_j}}(\eta\alpha - \alpha +1) E[Y_{\sigma_i}(t)]^{\eta -1}=E[Y_{\sigma_j}(t^{1/2}) \, Y_{\sigma_i}(t^{1/2})]^{\eta -1}.
\end{equation*}
From the fact that $\mathbb{P}_1 \ni G_\mu \sim g^1_{\mu} \in \mathbb{F}_1$ and 
\begin{equation*}
\mathbb{P}_{\gamma} \ni [G_\mu (t)]^{1/\gamma} \sim \tilde{g}^\gamma_\mu \quad  \Leftrightarrow \quad \tilde{g}^\gamma_\mu \in \mathbb{F}_{\gamma}
\end{equation*}
\begin{equation*}
\mathbb{P}_{1} \ni [G_\mu (t^\gamma) ]^{1/\gamma} \sim g^\gamma_\mu \quad  \Leftrightarrow \quad g^\gamma_\mu \in \mathbb{F}_{1}
\end{equation*}
we can define the following Mellin convolution of $g^\gamma_\mu \in \mathbb{M}_{1-\gamma \mu}^{\infty}$.
\begin{df}
For $\gamma \neq 0$, $\boldsymbol{\mu} = (\mu_1, \ldots , \mu_n) \in \mathbb{R}^n_{+}$ we define the Mellin convolution 
\begin{equation}
g^{\gamma, \star n}_{\boldsymbol{\mu}}(x, t) = g^{\gamma}_{\mu_1} \star \ldots \star g^{\gamma}_{\mu_n}(x, t) \label{GenConvMell}
\end{equation}
with Mellin transform (see \eqref{mellinTRF})
\begin{align}
\mathcal{M}[g^{\gamma, \star n}_{\boldsymbol{\mu}}(\cdot, t)] (\eta) = & \prod_{j=1}^n \mathcal{M}[g^{\gamma}_{\mu_j}(\cdot, t^{1/n})](\eta) =  t^{\eta -1} \prod_{j=1}^{n} \frac{\Gamma((\eta -1)/\gamma  + \mu_j )}{\Gamma(\mu_j)} \label{mellTransfDEF}
\end{align}
where $\eta \in \mathbb{H}^{1}_{a}$ and $a=1-\min_{j}\{\gamma \mu_j \}$.
\label{defMC}
\end{df}
We notice that $f_{X^\alpha} \star f_{Y^\beta} = f_{Y^\beta} \star f_{X^\alpha}$ is the law of $X^\alpha \cdot Y^\beta$  if $X \in \mathbb{P}_\alpha$ and $Y \in \mathbb{P}_\beta$ whereas the well-known Fourier convolution $f_X * f_Y$  is the law of $X + Y$. Also, we introduce the sets
\begin{equation}
\mathscr{S}^{n}_{\kappa} (\varsigma) = \left\lbrace \bar{\wp} \in \mathbb{R}^n_{+}\, : \, \bar{\wp}=\frac{\bar{\upsilon}}{\kappa},\, \bar{\upsilon}= (\upsilon_{1}, \ldots , \upsilon_{n}) \in \mathbb{N}^n,\; \sum_{j=1}^n \upsilon_j = \varsigma \right\rbrace
\end{equation}
and
\begin{align}
\mathscr{P}_{\kappa}^{n}\left( \varrho \right) = & \left\lbrace \bar{\wp} \in \mathbb{R}^{n}_{+}\,:\,  \bar{\wp} = \frac{\bar{\upsilon}}{\kappa},\, \bar{\upsilon} = (\upsilon_1 , \ldots , \upsilon_n) \in \mathbb{N}^{n}, \, \prod_{j=1}^{n} \upsilon_j = \varrho  \right\rbrace \label{setS}
\end{align}
with $m, \kappa, \varrho \in \mathbb{N}$. For $\gamma=1,2$ and a fixed $\boldsymbol{\mu} = (\mu_1, \ldots , \mu_n) \in \mathscr{S}^n_{\kappa}(\varsigma)$, we have that
\begin{equation*}
g^\gamma_{\mu_{1}} * \cdots * g^\gamma_{\mu_{n}}(x,t) = g^\gamma_{\theta_{\sigma_1}} * \cdots * g^\gamma_{\theta_{\sigma_n}} (x,t) 
\end{equation*}
for all $\boldsymbol{\theta}=(\theta_{\sigma_1}, \ldots , \theta_{\sigma_n}) \in \mathscr{S}^{n}_{\kappa} (\varsigma)$ and all permutations of $\{ \sigma_j\}$, $j=1,2, \ldots , n$. This fact follows easily from the semigroup property ($*$-commutativity) of the law $g^\gamma_\mu$ which will be shown in $ii)$, Lemma \ref{LemmaCONV} below. We observe that $ \aleph = | \mathscr{P}^{m}_{\kappa}| < | \mathbb{N} |$ is the cardinality of $\mathscr{P}^{m}_{\kappa}$, thus $\mathscr{P}^{m}_{\kappa}$ is a finite set. Furthermore, $\forall \varrho \in \mathbb{N}$ and a fixed $\boldsymbol{\mu} \in \mathscr{P}^{m}_{\kappa}(\varrho)$, we have that
\begin{equation}
\mathcal{M}[g^{\gamma, \star n}_{\boldsymbol{\mu}}(\cdot, t)](\eta) = \mathcal{M}[g^{\gamma, \star n}_{\boldsymbol{\theta}}(\cdot, t)](\eta), \quad \forall\, \boldsymbol{\theta} \in \mathscr{P}^{m}_{\kappa}(\varrho)
\end{equation}
whereas, for $\boldsymbol{\mu} \in \mathscr{S}^{m}_{\kappa}(\varsigma)$ and $\gamma=1,2$, we have that
\begin{equation}
\mathcal{F}[g^{\gamma, \ast n}_{\boldsymbol{\mu}}(\cdot, t)](\xi) = \mathcal{F}[g^{\gamma, \ast n}_{\boldsymbol{\theta}}(\cdot, t)](\xi), \quad \forall\, \boldsymbol{\theta} \in \mathscr{S}^{m}_{\kappa}(\varsigma)
\end{equation}
where we used, the familiar notation, $f^{\ast n}_{\boldsymbol{\mu}} = f_{\mu_1} * \cdots * f_{\mu_n}$. The symbols $\mathcal{M}$ and $\mathcal{F}$ stand for the Mellin and Fourier transforms.

\begin{te}
Let us consider $\nu=1/(n+1)$, $n \in \mathbb{N}$ and $\boldsymbol{\mu}=(\mu_1, \ldots , \mu_n)$. 
\begin{itemize}
\item [i)] For the stable subordinator $\mathfrak{H}^{\nu}_t$, $t>0$, the following equivalence in law holds true
\begin{equation}
\mathfrak{H}^{\nu}_t \stackrel{law}{=} E_{\mu_1}(E_{\mu_2}(\ldots E_{\mu_n}((\nu t)^{1/\nu})\ldots)), \quad t>0, \quad \boldsymbol{\mu} \in \mathscr{P}^{n}_{n+1}(n!) \label{sdkll10}
\end{equation}
where the process $E_{\mu}(t)$, $t>0$, satisfies the SDE \eqref{SPDEdiE}.
\item [ii)] For the inverse process $\mathfrak{L}^{\nu}_t$, $t>0$, the following equivalence in law holds true
\begin{equation}
\mathfrak{L}^{\nu}_t \stackrel{law}{=} \left[ G_{\mu_1}(G_{\mu_2}(\ldots G_{\mu_n}(t^\nu/\nu)\ldots)) \right]^{\nu}, \quad t>0, \quad \boldsymbol{\mu} \in \mathscr{P}^{n}_{n+1}(n!). \label{sdklI9}
\end{equation}
where the process $G_{\mu}(t)$, $t>0$, satisfies the SDE \eqref{SDEBes}.
\end{itemize}
\label{cSB}
\end{te}
From Theorem \ref{cSB} and formulas \eqref{gConv}, \eqref{eConv} we can explicitly write $\mathbf{g}^{\nu, \beta}_{\mu}$. Furthermore, this representation holds in the general set $\mathscr{P}^{n}_{n+1}(n!)$. Indeed, for $\nu =1/(n+1)$, $n \in \mathbb{N}$, the stochastic solution to \eqref{Mgdue} is given by
\[ G_\mu(E_{\mu_1}(E_{\mu_2}(\ldots E_{\mu_n}((\nu t)^{1/\nu})\ldots)), \quad t>0, \quad (\mu_1, \ldots , \mu_n) \in \mathscr{P}^{n}_{n+1}(n!). \]
Moreover, the representation \eqref{sdklI9} turns out to be useful in representing the solution to the problems  \eqref{tmpDG} and \eqref{probBoundG}. A natural extension follows for the problem \eqref{Mgtre}. From $E_\mu(ct) \stackrel{law}{=} \frac{1}{c}E_\mu$, $t>0$, $c >0$ and formula \eqref{sdkll10} we obtain that
\[ \mathfrak{H}^{\nu}_t \stackrel{law}{=} \nu^{-1/\nu} E_{\mu_1}(E_{\mu_2}(\ldots E_{\mu_n}(t^{1/\nu})\ldots)), \quad t>0 \]
and thus, for $\alpha_1=1/(n_1 +1)$, $\alpha_2=1/(n_2 +1)$ and
\[ \boldsymbol{\mu}_1 = (\mu_{1,1}, \ldots, \mu_{1,n_1}) \in \mathscr{P}^{n_1}_{n_1+1}(n_1!), \qquad  \boldsymbol{\mu}_2 = (\mu_{2,1}, \ldots, \mu_{2,n_2}) \in \mathscr{P}^{n_2}_{n_2+1}(n_2!) \]
we have that
\begin{equation} 
F^{\alpha_1, \alpha_2}_t \stackrel{law}{=}  \nu^{-1/\nu} E_{\mu_{1,1}}(E_{\mu_{1,2}}(\ldots E_{\mu_{1,n_1}}(G_{\mu_{1,2}}(G_{\mu_{2,2}}(\ldots G_{\mu_{2,n_2}}(t^\nu/\nu)\ldots)))\ldots)).\label{rapF}
\end{equation}
The fact that $E_{\mu}\in \mathbb{P}_1$ and $G_{\mu} \in \mathbb{P}_1$ means that 
\[ E_{\mu_1}(G_{\mu_2}(t)) \stackrel{law}{=} G_{\mu_2}(E_{\mu_1}(t)) \]
or equivalently
\[ g^{-1}_{\mu_1} \star g^{1}_{\mu_2} = g^{1}_{\mu_2} \star g^{-1}_{\mu_1}. \]
From this we can write the law \eqref{lawf} in terms of the convolutions \eqref{gConv} and \eqref{eConv}.
\begin{coro}
For $\nu =1/(n+1)$, $n \in \mathbb{N}$, the stochastic solution to \eqref{Mgtre} is given by 
\[ G_\mu(F^{\nu, \beta}_t), \quad t>0 \]
where $G_\mu$ has law $g^1_\mu$ and $F^{\nu, \beta}$ has density which can be represented by means of the Mellin convolutions  \eqref{gConv} and \eqref{eConv} as formula \eqref{rapF} entails.
\end{coro}

\begin{os}
\normalfont
For $\gamma \neq 0$, $\mu >0$ and $\nu=1/5$, the stochastic solution to \eqref{tmpDG} can be written as follows
\begin{equation*}
\left[ G_\mu(G_{\mu_1}(G_{\mu_2}(G_{\mu_3}(G_{\mu_4}(t^{1/5}/5))))) \right]^{1/\gamma}, \quad t>0
\end{equation*}
($G_{\mu}, G_{\mu_1}, \ldots, G_{\mu_4}$ independent squared Bessel processes) where 
$$\boldsymbol{\mu}=(\mu_1, \mu_2, \mu_3, \mu_4) \in \mathscr{P}^{4}_{5}(4!).$$ 
For $\boldsymbol{\mu} = (3,2,2,2)/5$, we have that
\begin{equation*}
\tilde{u}^{\gamma, \mu}_{1/5}(x,t) = C\,  \frac{x^{\gamma\mu -1}}{t^{2/5}} \int_0^\infty \int_0^\infty \frac{e^{- x^\gamma / z}}{s^{2/5}} K_{\frac{1}{5}}\left( 2\frac{z^{5/2}}{s} \right)\, K_0\left( \frac{2}{5^{3/2}}\frac{s}{t^{1/2}}\right)\, ds\, dz
\end{equation*}
where
\begin{equation*}
C = \frac{2 \, \Gamma\left( \frac{1}{5} \right) \Gamma\left( \frac{4}{5} \right)}{5^{3/2} \left[ \pi \Gamma\left( \frac{2}{5} \right) \right]^2 \Gamma\left(\frac{3}{2} \right)} .
\end{equation*}
For further configurations of $\boldsymbol{\mu}$, see Remark \ref{remark1quinto}.
\end{os}

\begin{os}
\normalfont
The relation between stable densities and higher order equations has been investigated by many authors (see for example \citet{BMN09b, DEBL04, Dov4, DO2}). In \cite{Dov4} we have shown that the law $l_{\frac{1}{n}}$ of $\mathfrak{L}^{\frac{1}{n}}_t$ solves the higher-order equation
\begin{equation*}
(-1)^n \frac{\partial^n u}{\partial x^{n}} = \frac{\partial u}{\partial t}.
\end{equation*}
\end{os}

\section{Auxiliary Results and Proofs}
\subsection{The operators $\mathcal{G}$ and $\mathcal{G}^{*}$}
The operators we deal with are given by
\begin{align}
\mathcal{G}^{*}\, f_2 = & \frac{1}{\gamma^2}\left(  \frac{\partial}{\partial x} x^{2-\gamma} \frac{\partial}{\partial x} - (\gamma \mu - 1) \frac{\partial}{\partial x} x^{1-\gamma} \right) f_2 \nonumber \\
= & \frac{1}{\gamma^2} \frac{\partial}{\partial x}\left( x^{\gamma \mu -\gamma + 1} \frac{\partial}{\partial x} \left( \frac{1}{\mathfrak{w}(x)}\,f_2 \right) \right), \quad f_2 \in D(\mathcal{G}^{*}) \label{AdjOpG}
\end{align}
and
\begin{align}
\mathcal{G}\, f_1 = & \frac{x^{1-\gamma}}{\gamma^2}\left( x \frac{\partial^2}{\partial x^2} + (\gamma \mu - \gamma +1) \frac{\partial}{\partial x} \right)  \, f_1 \nonumber \\
= & \frac{1}{\gamma^2\, \mathfrak{w}(x)} \frac{\partial}{\partial x} \left( x^{\gamma\mu -\gamma +1} \frac{\partial }{\partial x}  f_1 \right), \quad f_1 \in D(\mathcal{G}) \label{OpG}
\end{align}
where $\mathfrak{w}(x)=x^{\gamma\mu -1}$. We shall refer to $\mathcal{G}^{*}$ as the adjoint of $\mathcal{G}$. Indeed, as a straightforward check shows, we have that $\mathcal{G}^{*}\, \mathfrak{w}\, f_1 = \mathfrak{w}\, \mathcal{G}\, f_1$ and the Lagrange's identity 
\begin{equation}
f_2 \, \mathcal{G}\, f_1 - f_1\, \mathcal{G}^* \, f_2 = 0 \label{lagrangeIdentity}
\end{equation} 
immediately follows. Thus, by observing that
\begin{equation*}
D(\mathcal{G}^*) = \{ f \in \tilde{\mathbb{M}}_{1}: f = \mathfrak{w}\, f_1,\, f_1 \in D(\mathcal{G}) \},
\end{equation*}
(see Definition \ref{defM}) we obtain that 
\begin{equation*}
\langle \mathcal{G} f_{1}, \, f_{2} \rangle = \langle f_{1}, \, \mathcal{G}^{*} f_{2} \rangle, \quad \forall f_{1} \in D(\mathcal{G}) \textrm{ and } \forall f_{2} \in  D(\mathcal{G}^{*}).
\end{equation*}

\begin{lem}
The following hold true:
\begin{itemize}
\item [i)] For the operator appearing in \eqref{OpG} we have that
\begin{equation}
\mathcal{G} \, \psi_{\kappa} = (\kappa /2)^2 \, \psi_{\kappa} \label{AUTOG}
\end{equation}
where
\begin{equation}
\psi_{\kappa}(x) = x^{\frac{\gamma}{2}(1-\mu)} K_{\mu -1}\left( \kappa\, x^{\gamma/2}\right), \quad \kappa >0, \; x > 0, \; \gamma \neq 0 \label{psiK}
\end{equation}
and $K_{\alpha}$ is the Macdonald's function \eqref{functionK}.
\item [ii)] For the operator appearing in \eqref{OpG} we have that
\begin{equation}
\mathcal{G} \, \bar{\psi}_{\kappa} = - (\kappa /2)^2 \, \bar{\psi}_{\kappa} \label{AUTOGBAR}
\end{equation}
where
\begin{equation}
\bar{\psi}_{\kappa}(x) = x^{\frac{\gamma}{2}(1-\mu)} J_{\mu -1}\left( \kappa\, x^{\gamma/2}\right), \quad \kappa >0, \; x > 0, \; \gamma \neq 0 \label{psiK}
\end{equation}
and $J_{\alpha}$ is the Bessel function of the first kind \eqref{BesJ}.
\end{itemize}
\label{lmG}
\end{lem}
\begin{proof}
We first recall some properties of the Macdonald's function \eqref{functionK}: we will use the fact that $K_{-\alpha} = K_{\alpha}$ and
\begin{align}
\frac{d}{dz}K_{\alpha}(z) = - K_{\alpha -1}(z) - \frac{\alpha}{z} K_{\alpha}(z).
\end{align}
(see \cite[p. 110]{LE});  the functions $K_{\alpha}$ and $I_{\alpha}$ are two linearly independent solutions of the Bessel equation
\begin{equation}
x^2 \frac{d^2 Z_{\alpha}(x)}{dx^2} + x \frac{d Z_{\alpha}(x)}{dx} - x^2 Z_{\alpha}(x)=0   \label{eqBesuno} 
\end{equation}
whereas, the functions $J_{\alpha}$ and $Y_{\alpha}$ (see \cite{LE} for definition) are linearly independent solutions to
\begin{equation}
x^2 \frac{d^2 Z_{\alpha}(x)}{dx^2} + x \frac{d Z_{\alpha}(x)}{dx} + x^2 Z_{\alpha}(x)=0  \label{eqBesdue}
\end{equation} 
(see \cite[pp. 105 - 110]{LE}). 

By performing the first and the second derivative with respect to $x$ of the function $\psi_{\kappa} = \psi_{\kappa}(x)$ we obtain
\begin{align*}
\psi_{\kappa}^{\prime} =  & \frac{\gamma}{2}(1-\mu)\frac{1}{x} \psi_{\kappa} + x^{\frac{\gamma}{2}(1-\mu)} \frac{\gamma \kappa}{2x} x^{\gamma/2} \left[ -K_{-\mu} - \frac{1-\mu}{\kappa x^{\gamma /2}} K_{1-\mu} \right] \\
= & \frac{\gamma}{2}(1-\mu)\frac{1}{x} \psi_{\kappa} - \frac{\gamma \kappa}{2x} x^{\frac{\gamma}{2}(2-\mu)} K_{-\mu} - \frac{\gamma}{2}(1-\mu)\frac{1}{x} \psi_{\kappa} = - \frac{\gamma \kappa}{2x} x^{\frac{\gamma}{2}(2-\mu)} K_{-\mu}
\end{align*}
and
\begin{align*}
\psi_{\kappa}^{\prime \prime}= & \left( \frac{\gamma}{2}(2-\mu) -1 \right) \frac{1}{x} \psi_{\kappa}^\prime + \frac{\gamma \kappa}{2} x^{\frac{\gamma}{2}(2-\mu) -1} \frac{\gamma \kappa}{2x} x^{\gamma/2} \left[ - K_{\mu -1} - \frac{\mu}{\kappa x^{\gamma /2}} K_{\mu} \right]\\
= & \left( \frac{\gamma}{2}(2-\mu) -1 \right) \frac{1}{x} \psi_{\kappa}^\prime - \left( \frac{\gamma \kappa}{2} \right)^2 \frac{x^{\frac{\gamma}{2} (1-\mu) + \gamma}}{x^2} K_{\mu -1} - \frac{\gamma \mu}{2x} \psi_{\kappa}^\prime.
\end{align*}
By keeping in mind the operator $\mathcal{G}$, from the fact that
\begin{align}
x \psi_{\kappa}^{\prime \prime} + (\gamma \mu - \gamma +1) \psi_{\kappa}^\prime = & x^{\gamma -1} \frac{\gamma^2 \kappa^2}{2^2} \psi_{\kappa} \label{AUTOGtmp}
\end{align}
the relation \eqref{AUTOG} is obtained. The equation \eqref{AUTOGtmp} can be rewritten as
\begin{align}
x^2\, \psi_{\kappa}^{\prime \prime} + (\gamma \mu -\gamma +1) \, x \, \psi_{\kappa}^{\prime} - \gamma^2\, (\kappa/2)^2 \, x^{\gamma}\, \psi_{\kappa} = 0 \label{eqBessGen}
\end{align}
which is related to the formula \eqref{eqBesuno} whereas, a slightly modified version of \eqref{eqBessGen}, which is
\begin{align}
x^2\, \bar{\psi}_{\kappa}^{\prime \prime} + (\gamma \mu -\gamma +1) \, x \, \bar{\psi}_{\kappa}^{\prime} + \gamma^2\, (\kappa/2)^2 \, x^{\gamma}\, \bar{\psi}_{\kappa} = 0, \label{eqBessGen2}
\end{align}
is related to the formula \eqref{eqBesdue}. The equation \eqref{AUTOGBAR} can be written as formula \eqref{eqBessGen2} and therefore, after some algebra, from \eqref{eqBesdue}, we have at once that 
\begin{equation*}
\bar{\psi}_{\kappa}(x) = x^{\frac{\gamma}{2}(1 - \mu)} J_{\mu -1}\left( \kappa\, x^{\gamma/2}\right)
\end{equation*}
as announced in the statement of the Lemma.
\end{proof}

Formula \eqref{eqBessGen2} can be put into the Sturm-Liouville form as follows
\begin{equation}
\left( x^{\gamma \mu - \gamma +1} \bar{\psi}_{\kappa}^{\prime} \right)^{\prime} + \gamma^2\, (\kappa/2)^2 \, \mathfrak{w}(x)\, \bar{\psi}_{\kappa} = 0. \label{eqBessGGG}
\end{equation}
According to the Sturm-Liouville theory (\cite{CouHilb}) and formula \eqref{eqBessGGG}, we obtain an orthogonal system $\{ \bar{\psi}_{\kappa_i} \}_{i \in \mathbb{N}}$ such that
\begin{equation}
\mathcal{G}\, \bar{\psi}_{\kappa_i} = - (\kappa_i/2)^2 \, \bar{\psi}_{\kappa_i}, \label{eqBGGG}
\end{equation}
where $\kappa_i$ are  the zeros of $J_{\alpha}$ and $\mathcal{G}$ is a Hermitian linear operator whose eigenfunctions are orthogonal w.r.t. the weight function $\mathfrak{w}(x) = x^{\gamma \mu -1}$. Indeed, from the fact that
\begin{equation*}
\int_0^1 x \, J_\nu(\kappa_i x)\, J_\nu(\kappa_j x)\, dx = 0, \quad \textrm{if }\; i \neq j
\end{equation*}
(see \cite{LE}) we get that
\begin{equation}
\int_{0}^{1}  \bar{\psi}_{\kappa_i}(x)  \bar{\psi}_{\kappa_j}(x)  \mathfrak{w}(x) dx = 0, \quad \textrm{if }\; i \neq j. \label{innerprod}
\end{equation}

\subsection{Proof of Theorem 1 : time-fractional diffusions in one-dimensional half-space. }
For $\nu=1$ the density law \eqref{tmpDGlaw} becomes the law of $G_\mu$,  $\tilde{u}^{\gamma, \mu}_{1} = \tilde{g}^{\gamma}_{\mu}$ whose Mellin transform is written as
\begin{equation}
\Psi_t(\eta) = \mathcal{M}[\tilde{g}^{\gamma}_{\mu}(\cdot , t)](\eta)= \Gamma\left( \frac{\eta -1}{\gamma} + \mu \right) \frac{t^\frac{\eta -1}{\gamma}}{\Gamma(\mu)}, \quad \eta \in \mathbb{H}_{1-\gamma \mu}^{\infty} .
\label{Zyyy}
\end{equation}
We perform the time derivative of \eqref{Zyyy} and obtain
\begin{align*}
\frac{\partial}{\partial t} \Psi_t(\eta) = &\frac{\eta -1}{\gamma} \Gamma\left( \frac{\eta -1}{\gamma} + \mu \right) t^\frac{\eta -\gamma -1}{\gamma}\\
=  & \frac{\eta -1}{\gamma} \left( \frac{\eta - \gamma -1 + \gamma \mu}{\gamma} \right) \Gamma\left( \frac{\eta - \gamma -1}{\gamma} + \mu \right) t^\frac{\eta - \gamma -1}{\gamma}\\
= & \frac{1}{\gamma^2} (\eta -1) ( \eta - \gamma -1 + \gamma \mu ) \Psi_t(\eta - \gamma)\\
= & \frac{1}{\gamma^2} (\eta -1) ( \eta - \gamma) \Psi_t(\eta - \gamma) +  \frac{1}{\gamma^2} (\eta -1) ( \gamma \mu -1 ) \Psi_t(\eta - \gamma)\\
= & \frac{1}{\gamma^2} \mathcal{M}\left[ \frac{\partial}{\partial x} x^{2-\gamma} \frac{\partial}{\partial x} \tilde{g}^{\gamma}_{\mu} \right] (\eta) - \frac{(\gamma \mu -1)}{\gamma^2} \mathcal{M} \left[ \frac{\partial}{\partial x} x^{1-\gamma} \tilde{g}^{\gamma}_{\mu} \right] (\eta).
\end{align*}
From the fact that $\tilde{g}^\gamma_\mu \in \tilde{\mathbb{M}}_{1}$ and according to the properties \eqref{derMint}, \eqref{propM1} and \eqref{propM2}, the inverse Mellin transform yields the claimed result. We give an alternative proof by exploiting the Laplace transform technique. The Laplace transform of $\tilde{g}^{\gamma}_{\mu}(x,t)$, $x,t>0$, can be evaluated by recalling that (see \cite[formula 3.478]{GR})
\begin{equation}
\int_0^\infty x^{\nu -1} \exp \left\lbrace -\beta x^p - \gamma x^{-p} \right\rbrace dx = \frac{2}{p} \left( \frac{\gamma}{\beta} \right)^\frac{\nu}{2p} K_\frac{\nu}{p} \left(2 \sqrt{\gamma \beta} \right) \label{formula:K}
\end{equation}
where $p,\gamma,\beta,\nu >0$ and  $K_\nu$ is the modified Bessel function. Thus, we obtain
\begin{align*}
\mathcal{L}[\tilde{g}^{\gamma}_{\mu}(x, \cdot)](\lambda) = & 2\frac{x^{\frac{\gamma}{2}(\mu +1) - 1}}{\Gamma(\mu) \lambda^{\frac{1-\mu}{2}}} K_{1-\mu} \left( 2 \lambda^{1/2} x^{\gamma /2} \right) = 2 \frac{\mathfrak{w}(x)}{\Gamma(\mu)} \, f(\lambda) \, \psi(x; 2\lambda^{1/2})
\end{align*}
where $f(\lambda) = \lambda^{(\mu -1)/2}$ and $\psi_{\kappa}(x) = \psi(x; \kappa)$ is that in Lemma \ref{lmG}. By considering that $\tilde{g}^{\gamma}_{\mu} (x,t) = \mathfrak{w}(x) \tilde{k}^{\gamma}_{\mu}(x,t)$ and $\mathcal{G}^{*} \, \mathfrak{w}(x) \tilde{k}^{\gamma}_{\mu} = \mathfrak{w}(x)\, \mathcal{G} \, \tilde{k}^{\gamma}_{\mu}$ we get that
\begin{align*}
\mathcal{L}[\mathcal{G}^{*} \, \tilde{g}^{\gamma}_{\mu}(x, \cdot)](\lambda) = & 2 \frac{\mathfrak{w}(x)}{\Gamma(\mu)} \, f(\lambda) \, \mathcal{G} \, \psi(x; 2\lambda^{1/2}) = \lambda \, \mathcal{L}[\tilde{g}^{\gamma}_{\mu}(x, \cdot)](\lambda)
\end{align*}
where in the last formula we used the result \eqref{AUTOG}. From the fact that
\begin{equation*}
\mathcal{L}\left[ \frac{\partial}{\partial t} \tilde{g}^{\gamma}_{\mu}(x, \cdot) \right](\lambda) = \lambda \mathcal{L}[\tilde{g}^{\gamma}_{\mu}(x, \cdot)](\lambda), \quad x>0
\end{equation*}
we obtain the claimed result for $\nu=1$.\\ 

Now, we consider $\nu \in (0,1)$. From the Laplace transform 
\[ \mathcal{L}[l_{\nu}(x, \cdot)](\lambda) = \lambda^{\nu -1} \exp(-x \lambda^{\nu}) \] 
(see formula \eqref{LapTime}) we obtain that
\begin{align*}
\mathcal{L}[ \tilde{u}^{\gamma, \mu}_{\nu}(x, \cdot) ](\lambda) = & \int_0^\infty \tilde{g}^{\gamma}_{\mu}(x, s) \, \mathcal{L}[l_{\nu}(s, \cdot)](\lambda) \, ds\\
= & 2 \frac{\mathfrak{w}(x)}{\Gamma(\mu)}  \frac{\lambda^{\nu -1}}{\lambda^{\frac{\nu}{2}(1 - \mu )}} x^{\frac{\gamma}{2}(1-\mu)} K_{1-\mu}\left( 2 x^{\gamma/2} \lambda^{\nu /2} \right)\\ 
= & 2 \frac{\mathfrak{w}(x)}{\Gamma(\mu)} f(\lambda) \, \psi_{\kappa}(x)
\end{align*}
where $\psi_{\kappa}(x)=\psi(x;\kappa)$ is that in \eqref{psiK} with $\kappa=2\lambda^{\nu/2}$ and $f(\lambda) = \lambda^{\nu(\mu+1)/2 - 1}$. Thus, in the right-hand side of \eqref{tmpDG} we obtain
\begin{align*}
\mathcal{L}[\mathcal{G}^{*}\, \tilde{u}^{\gamma, \mu}_{\nu}(x, \cdot)](\lambda) = & \frac{2}{\Gamma(\mu)} f(\lambda) \, \mathcal{G}^{*} \,\mathfrak{w}(x) \, \psi(x;2\lambda^{\nu/2}) = 2 \frac{\mathfrak{w}(x)}{\Gamma(\mu)} f(\lambda)  \, \mathcal{G} \, \psi(x;2\lambda^{\nu/2})
\end{align*}
where we have used the fact that $\mathcal{G}^{*} \, \mathfrak{w} \, f = \mathfrak{w} \, \mathcal{G} \, f$. 
Finally, from \eqref{AUTOG}, we obtain
\begin{equation}
\mathcal{L}[\mathcal{G}^{*}\, \tilde{u}^{\gamma, \mu}_{\nu}(x, \cdot)](\lambda) = \lambda^{\nu}\, \mathcal{L}[\tilde{u}^{\gamma, \mu}_{\nu}(x, \cdot)](\lambda). \label{destro}
\end{equation}
We note that $|\tilde{u}^{\gamma, \mu}_{\nu}(\cdot, t)| \leq Be^{-q_0 t}$ for some $B, q_0 >0$ as a function of $t$ and thus, 
\begin{equation} 
\mathcal{L}\left[\frac{\partial^\nu \tilde{u}^{\gamma, \mu}_{\nu}}{\partial t^\nu} (x, \cdot)\right](\lambda) = \lambda^{\nu} \mathcal{L}[\tilde{u}^{\gamma, \mu}_{\nu}(x, \cdot)](\lambda),\label{sinistro}
\end{equation}
see \cite[Lemma 2.14]{KST06}. By comparing \eqref{destro} with \eqref{sinistro} the result follows.

\subsection{Proof of Theorem 2 : regular Sturm-Liouville problems.}

From the fact that $\mathcal{G}^{*}\, m^{\gamma, \mu}_{\nu}(x,t) = \mathfrak{w}(x) \mathcal{G}\, \bar{m}^{\gamma, \mu}_{\nu}(x,t)$, the problem \eqref{probBoundG} reduces to
\begin{equation}
\frac{\partial^\nu}{\partial t^\nu} \,\bar{m}^{\gamma, \mu}_{\nu} = \mathcal{G} \, \bar{m}^{\gamma, \mu}_{\nu}, \quad m^{\gamma,\mu}_{\nu}(\partial \Omega_{1},t)=0,\quad \bar{m}^{\gamma, \mu}_{\nu}(x,0)=m_0(x)/\mathfrak{w}(x) \label{qwmPDE}
\end{equation}
where $m^{\gamma, \mu}_{\nu}(x,t) = \mathfrak{w}(x) \, \bar{m}^{\gamma, \mu}_{\nu}(x,t)$ and $\mathfrak{w}(x) = x^{\gamma \mu -1}$. Furthermore, from Lemma \ref{lmG}, we have that $\mathcal{G}\, \bar{\psi}_{\kappa_i} = - (\kappa_i /2)^2 \bar{\psi}_{\kappa_i}$ where
\begin{equation}
\bar{\psi}_{\kappa_i}(x) = x^{\frac{\gamma}{2}(1 - \mu)} J_{\mu -1}\left(\kappa_i x^{\gamma/2} \right)
\end{equation}
and $\kappa_i$, $i \in \mathbb{N}$ are the zeros of $J_{\alpha}$. Formula \eqref{innerprod} leads to the orthonormal system $ \left\lbrace \bar{\psi}_{\kappa_i}(x) / \| \bar{\psi}_{\kappa_i} \|^2_{w};\, i \in \mathbb{N} \right\rbrace$ where $\| f\|^2_{\mathfrak{w}} = \langle f,\,f \rangle_{\mathfrak{w}}$ is the norm associated to the inner product \eqref{innerprod} with respect to the weight function $\mathfrak{w}(x)$. Thus, 
\[ L^2(\mathbb{R}) = \bigoplus_{n=1}^{\infty} \mathcal{H}_n \]
where $\mathcal{H}_n$ is the space of eigenfunctions associated with the eigenvalue $\lambda_n = (\kappa_n/2)^2$ and we obtain that
\begin{equation}
\bar{m}^{\gamma, \mu}(x,t) = \sum_{n=1}^{\infty} c_n ( t,\, \lambda_n) \, \frac{\bar{\psi}_{\kappa_n}(x)}{\| \bar{\psi}_{\kappa_n} \|^2_{\mathfrak{w}}} \label{frtSol}
\end{equation}
where $ \| \bar{\psi}_{\kappa_n} \|_{\mathfrak{w}} = J^{\prime}_{\mu -1}(\kappa_{n})  / \sqrt{\gamma }$ (see, e.g.\ \cite[p. 130]{LE}). From \eqref{qwmPDE} we have that
\begin{equation}
\sum_{n=1}^{\infty} \frac{\partial^\nu}{\partial t^\nu}\, c_n ( t,\, \lambda_n) \, \frac{\bar{\psi}_{\kappa_n}(x)}{\| \bar{\psi}_{\kappa_n} \|^2_{\mathfrak{w}}} = \sum_{n=1}^{\infty} c_n ( t,\, \lambda_n) \, \mathcal{G}\, \frac{\bar{\psi}_{\kappa_n}(x)}{\| \bar{\psi}_{\kappa_n} \|^2_{\mathfrak{w}}} \label{suggestsPDE}
\end{equation}
which holds term by term. From the fact that 
\[ \mathcal{G}\, \bar{\psi}_{\kappa_n}(x)  = - \lambda_n \, \bar{\psi}_{\kappa_n}(x) \] 
where $\lambda_n = (\kappa_n /2)^2$ (see \eqref{AUTOGBAR}), formula \eqref{suggestsPDE} lead to the fractional equation
\begin{equation*}
{^cD^\nu_t}\, c_n ( t,\, \lambda_n) = - \lambda_n\, c_n ( t,\, \lambda_n)
\end{equation*}
and thus, we obtain
\begin{equation}
c_n(t,\, \lambda_n) = c_{n} \cdot E_{\nu}(-\lambda_n\, t^{\nu})
\end{equation}
(the Mittag-Leffler is an eigenfunction of the Dzhrbashyan-Caputo fractional derivative) where $c_n$ must be determined by taking into account the initial data. In particular,
\begin{equation*}
c_n = \langle m_0/\mathfrak{w},\, \bar{\psi}_{\kappa_n} \rangle_{\mathfrak{w}} = \int_{\Omega_1} m_0(x) \bar{\psi}_{\kappa_n}(x) \,dx.
\end{equation*}
Formula \eqref{frtSol} solves \eqref{qwmPDE} and we obtain
\begin{equation}
\bar{m}^{\gamma, \mu}_{\nu}(x,t) = \langle \bar{m}^{\gamma, \mu}(x,\cdot), \, l_{\nu}(\cdot, t) \rangle = \sum_{n} c_n \,E_{\nu}\left(-(\kappa_n /2)^2\, t^{\nu}\right) \, \frac{\bar{\psi}_n(x)}{\| \bar{\psi}_{\kappa_n} \|^2_{\mathfrak{w}}}.
\end{equation} 
We have to observe that $m^{\gamma, \mu}_{\nu}(x,t) = \mathfrak{w}(x) \bar{m}^{\gamma, \mu}_{\nu}(x,t)$ for the proof to be completed.

\subsection{Proof of Theorem 3}
For $\mu >0$, $\alpha \in (0,1)$ 
\begin{equation}
\exists\, \mathbb{S} \subset \mathbb{H}_{0}^{\infty} \; \textrm{s.t.} \; (\mathcal{T} I_{0-}^{1-\alpha} k^1_\mu)(\eta) =0, \quad \eta \in \mathbb{S} \label{Trond}
\end{equation}
(see the Appendix \ref{appeA}) where $g^\gamma_\mu(x,t) = \mathfrak{w}(x)\, k^\gamma_\mu(x,t)$ and $k^{\gamma}_{\mu}(x,t) = |\gamma|/ \Gamma(\mu) \exp(-(x/t)^\gamma ) / t^{\gamma \mu}$. Indeed, being
\begin{equation*}
(I_{0-}^{1-\alpha} f)(x) = \frac{1}{\Gamma(\alpha)} \int_x^\infty (s-x)^{\alpha -1} f(s)\, ds, \quad x>0
\end{equation*}
we obtain 
\begin{align*}
(I_{0-}^{1-\alpha} k^1_\mu(\cdot, t))(x) = & t^{\alpha -1}\, k^1_{\mu}(x,t)
\end{align*}
and \eqref{Trond} immediately follows. We restrict ourselves to the case $\nu \in (0,1)$. From the formula \eqref{derMfracmeno}, we obtain
\begin{align*}
& \int_0^\infty x^{\eta -1} \frac{\partial^\nu}{\partial (-x)^\nu} \bigg( x^{\mu - 1 +\nu} \frac{\partial^\nu}{\partial (-x)^\nu} \Big( x^{1 - \mu}\, g^{1}_{\mu}(x,t) \Big) \bigg)\, dx\\
= & \frac{\Gamma(\eta)}{\Gamma(\eta -\nu)} \int_0^\infty x^{\eta -\nu -1} \bigg( x^{\mu - 1 +\nu} \frac{\partial^\nu}{\partial (-x)^\nu} \Big( x^{1 - \mu}\, g^{1}_{\mu}(x,t) \Big) \bigg)\, dx\\
= & \frac{\Gamma(\eta)}{\Gamma(\eta -\nu)} \int_0^\infty x^{(\eta + \mu -1)-1} \frac{\partial^\nu}{\partial (-x)^\nu} \Big( x^{1 - \mu}\, g^{1}_{\mu}(x,t) \Big) \, dx\\
= & \frac{\Gamma(\eta)}{\Gamma(\eta -\nu)} \frac{\Gamma(\eta + \mu -1)}{\Gamma(\eta + \mu -1 - \nu)} \int_0^\infty x^{(\eta + \mu -1- \nu)-1}\Big( x^{1 - \mu}\, g^{1}_{\mu}(x,t) \Big) \, dx\\
= & \frac{\Gamma(\eta)}{\Gamma(\eta -\nu)} \frac{\Gamma(\eta + \mu -1)}{\Gamma(\eta + \mu -1 - \nu)} \int_0^\infty x^{(\eta - \nu)-1} g^{1}_{\mu}(x,t) \, dx\\
=& \frac{\Gamma(\eta)}{\Gamma(\eta -\nu)} \frac{\Gamma(\eta + \mu -1)}{\Gamma(\eta + \mu -1 - \nu)} \mathcal{M}[g^{1}_{\mu}(\cdot ,t)](\eta - \nu).
\end{align*}
The $x$-Mellin transform of both members of \eqref{Mguno} writes
\begin{equation*}
\frac{\partial^\nu}{\partial t^\nu} \, \mathcal{M}[g^1_{\mu}(\cdot, t)](\eta) = \frac{\Gamma(\eta)}{\Gamma(\eta -\nu)} \frac{\Gamma(\eta +\mu -1)}{\Gamma(\eta +\mu -1 - \nu)} \mathcal{M}[g^{1}_{\mu}(\cdot, t)](\eta -\nu)
\end{equation*}
where $\mathcal{M}[g^1_\mu(\cdot, t)](\eta) = t^{\eta -1}\Gamma(\eta +\mu -1)/\Gamma(\mu)$, $\eta \in \mathbb{H}_{1-\mu}^{\infty}$. Thus, we have that
\begin{equation*}
\frac{\partial^\nu}{\partial t^\nu}\, \mathcal{M}[g^1_{\mu}(\cdot, t)](\eta) = \frac{\Gamma(\eta)}{\Gamma(\eta -\nu)} \frac{\Gamma(\eta +\mu -1)}{\Gamma(\mu)} t^{\eta -1 -\nu}
\end{equation*}
because of the fact that $\frac{\partial^\beta}{\partial t^\beta} t^{\beta-1} = \Gamma(\beta)/\Gamma(\beta-\alpha) t^{\beta - \alpha -1}$ (see e.g. \cite[Property 2.5]{SKM93}) which concludes the proof.

\subsection{Proof of Theorem 4}

We proceed as follows: first of all we find out the Mellin transform of the fractional operator acting on space
\begin{equation}
\mathcal{A} f(x, t) = -\frac{\partial^\nu}{\partial x^\nu} \bigg( x^{\mu - 1 +\nu} \frac{\partial^\nu}{\partial (-x)^\nu} \Big( x^{1-\mu}\, f(x,t) \Big) \bigg), \quad \nu \in (0,1) \label{AopF}
\end{equation}
for a well-defined function $f \in \mathbb{M}_a^\infty$, $a \in \mathbb{R}$ (and for which \eqref{Trond} holds, that is $(\mathcal{T}I^{1-\alpha}_{0\pm} f)(\eta) = 0$) and second of all we prove $ii)$ by exploiting the Mellin technique and then $i)$ as a particular case of $ii)$. We also consider $f \in D(\mathcal{G}^*)$. Let us write 
\begin{equation*}
\Phi_{\nu}(\eta) = \frac{\Gamma(1-\eta +\nu)}{\Gamma(1-\eta)} \quad \textrm{and} \quad \Psi_{\nu}(\eta) = \frac{\Gamma(\eta + \mu -1)}{\Gamma(\eta +\mu -1 - \nu)}.
\end{equation*}
From \eqref{derMfrac2} we have that
\begin{align*}
\int_0^\infty x^{\eta -1} \mathcal{A}f(x,t) dx = & - \Phi_{\nu}(\eta) \int_0^\infty x^{\eta -\nu -1} x^{\mu - 1 +\nu} \frac{\partial^\nu}{\partial (-x)^\nu} \Big( x^{1-\mu}\, f(x,t) \Big) dx\\
= & - \Phi_{\nu}(\eta) \int_0^\infty x^{\eta + \mu - 1} \frac{\partial^\nu}{\partial (-x)^\nu} \Big( x^{1-\mu}\, f(x,t) \Big) dx.
\end{align*}
From \eqref{derMfracmeno} we obtain  
\begin{align*}
\int_0^\infty x^{\eta -1} \mathcal{A}f(x,t) dx = & - \Phi_{\nu}(\eta) \Psi_{\nu}(\eta) \int_0^\infty x^{\eta +\mu - \nu - 1} x^{1-\mu} f(x,t) dx\\
= & - \Phi_{\nu}(\eta)  \Psi_{\nu}(\eta) \mathcal{M}[f(\cdot, t)](\eta -\nu).
\end{align*}
Thus, by collecting all pieces together we have that
\begin{equation}
\mathcal{M}[\mathcal{A}f(\cdot, t)](\eta) = - \frac{\Gamma(1-\eta +\nu)}{\Gamma(1-\eta)}\frac{\Gamma(\eta + \mu -1)}{\Gamma(\eta +\mu -1 - \nu)} \mathcal{M}[f(\cdot, t)](\eta -\nu). \label{melArond}
\end{equation}
Now, we consider the $x$-Mellin transform
\begin{align*}
\mathcal{M}[\mathbf{g}^{\nu, \beta}_{\mu}(\cdot, t)](\eta) = & \mathcal{M}[g^1_{\mu}(\cdot, 1)](\eta) \times \mathcal{M}[\mathfrak{f}_{\nu, \beta}(\cdot, t)](\eta)\\
= & \frac{\Gamma(\eta +\mu -1)}{\Gamma(\mu)} \mathcal{M}[\mathfrak{f}_{\nu, \beta}(\cdot, t)](\eta)
\end{align*}
where the fact that $h_{\nu} \in \mathbb{F}_{\nu}$ and $l_{\beta} \in \mathbb{F}_{1/\beta}$ leads to
\begin{equation*}
\mathcal{M}[\mathfrak{f}_{\nu, \beta}(\cdot, t)](\eta) = \mathcal{M}[h_{\nu}(\cdot, 1)](\eta) \times \mathcal{M}[l_{\beta}(\cdot, t)]\left( \frac{\eta -1}{\nu} +1 \right).
\end{equation*}
and, from the formulas \eqref{melTUTTE} we obtain
\begin{equation}
\mathcal{M}[\mathbf{g}^{\nu, \beta}_{\mu}(\cdot, t)](\eta) = \frac{\Gamma(\eta +\mu -1)}{\Gamma(\mu)} \frac{\Gamma\left( \frac{1-\eta}{\nu} \right)}{\nu\, \Gamma(1-\eta)} \frac{\Gamma\left( \frac{\eta -1}{\nu} +1 \right)}{\Gamma\left(\frac{\eta -1}{\nu}\beta +1 \right)} t^{\frac{\eta -1}{\nu} \beta}, \quad \eta \in \mathbb{H}_{a}^{1} \label{MelGProof}
\end{equation}
where $a = \max \{0, 1-\mu \}$, $\mu >0$. Now, we show that
\begin{equation}
\frac{\partial^\beta}{\partial t^\beta} \mathcal{M}[\mathbf{g}^{\nu, \beta}_{\mu}(\cdot, t)](\eta) = \mathcal{M}[\mathcal{A}\, \mathbf{g}^{\nu, \beta}_{\mu}(\cdot, t)](\eta) \label{meDAtmp}
\end{equation}
by taking into account the formula \eqref{melArond}. The right-hand side of the formula \eqref{meDAtmp} can be written as 
\begin{align*}
\mathcal{M}[\mathcal{A}\, \mathbf{g}^{\nu, \beta}_{\mu}(\cdot, t)](\eta) = & - \frac{\Gamma(1 -\eta +\nu)}{\Gamma(1-\eta)} \frac{\Gamma(\eta +\mu -1)}{\Gamma(\eta +\mu -\nu -1)} \mathcal{M}[\mathbf{g}^{\nu, \beta}_{\mu}(\cdot, t)](\eta - \nu)\\
= & - \frac{\Gamma(\eta +\mu -1)}{\Gamma(\mu)} \frac{\left( \frac{1-\eta}{\nu} \right) \Gamma\left( \frac{1-\eta}{\nu} \right)}{\nu\, \Gamma(1-\eta)} \frac{\Gamma\left( \frac{\eta -1}{\nu}\right)}{\Gamma\left(\frac{\eta -1}{\nu}\beta - \beta +1 \right)} t^{\frac{\eta -1}{\nu} \beta - \beta}\\
=& \frac{\Gamma(\eta +\mu -1)}{\Gamma(\mu)} \frac{\Gamma\left( \frac{1-\eta}{\nu} \right)}{\nu\, \Gamma(1-\eta)} \frac{\Gamma\left( \frac{\eta -1}{\nu} + 1 \right)}{\Gamma\left(\frac{\eta -1}{\nu}\beta - \beta +1 \right)} t^{\frac{\eta -1}{\nu} \beta - \beta}\\
= & \mathcal{M}[\mathbf{g}^{\nu, \beta}_{\mu}(\cdot, 1)](\eta) \frac{\Gamma\left(\frac{\eta -1}{\nu}\beta +1 \right)}{\Gamma\left(\frac{\eta -1}{\nu}\beta - \beta +1 \right)} t^{\frac{\eta -1}{\nu} \beta - \beta}.
\end{align*}
From the fact that
\begin{equation*}
\frac{\partial^\beta}{\partial t^\beta} \, t^{\frac{\eta -1}{\nu}\beta} = \frac{\Gamma\left( \frac{\eta -1}{\nu}\beta +1 \right)}{\Gamma\left( \frac{\eta -1}{\nu}\beta -\beta +1 \right)} t^{\frac{\eta -1}{\nu}\beta -\beta}
\end{equation*}
(see \cite[Property 2.5]{SKM93}) formula \eqref{meDAtmp} immediately follows and this prove $ii)$. \\

For $\beta=1$, the formula \eqref{MelGProof} takes the form 
\begin{align*}
\mathcal{M}[\mathfrak{g}^{\nu}_{\mu}(\cdot, t)](\eta) = & \frac{\Gamma(\eta +\mu -1)}{\nu \, \Gamma(\mu)\, \Gamma(1-\eta)} \Gamma\left( \frac{1-\eta}{\nu} \right) t^{\frac{\eta - 1}{\nu}}, \quad \eta \in \mathbb{H}_a^{1}
\end{align*}
where $a=\max\{0, 1-\mu \}$. This is (for $\beta=1$) because of the fact that $\mathfrak{H}^{\nu}_{\mathfrak{L}^{1}_t} \stackrel{a.s.}{=} \mathfrak{H}^{\nu}_t$, $t>0$, being $\mathfrak{L}^{1}_t \stackrel{a.s.}{=} t$ the elementary subordinator, see e.g. \citet{Btoi96}. Thus, form \eqref{melArond}, the Mellin transform of both members of \eqref{Mgdue} becomes 
\begin{align*}
- \frac{\partial}{\partial t} \mathcal{M}[\mathfrak{g}^{\nu}_{\mu}(\cdot, t)](\eta) = \frac{\Gamma(1-\eta + \nu)}{\Gamma(1-\eta)} \frac{\Gamma(\eta +\mu -1)}{\Gamma(\eta +\mu -1 - \nu)} \mathcal{M}[\mathfrak{g}^{\nu}_{\mu}(\cdot, t)](\eta -\nu)
\end{align*}
where
\begin{align*}
\mathcal{M}[\mathfrak{g}^{\nu}_{\mu}(\cdot, t)](\eta -\nu) = & \frac{\Gamma(\eta +\mu -1 - \nu)}{\pi \nu \, \Gamma(\mu)\, \Gamma(1 - \eta + \nu)} \Gamma\left( \frac{1-\eta + \nu}{\nu} \right) t^{\frac{\eta -\nu -1}{\nu}}\\
= & -\left(\frac{\eta -1}{\nu} \right) \frac{\Gamma(\eta +\mu -1 - \nu)}{\pi \nu \, \Gamma(\mu)\, \Gamma(1 - \eta + \nu)} \Gamma\left( \frac{1-\eta}{\nu} \right) t^{\frac{\eta -1}{\nu} - 1}\\
= & - \frac{\Gamma(\eta +\mu -1 - \nu)}{\pi \nu \, \Gamma(\mu)\, \Gamma(1 - \eta + \nu)} \Gamma\left( \frac{1-\eta}{\nu} \right) \frac{\partial}{\partial t} t^{\frac{\eta -1}{\nu}}.
\end{align*}
By collecting all pieces together we obtain the result claimed in $i)$.\\

From \eqref{MelGProof} and by direct inspection of \eqref{mellinHfox} we arrive at
\begin{equation*}
\mathcal{M}[\mathbf{G}^{\nu, \beta}_{\mu}(\cdot)](\eta) = \mathcal{M}^{2,1}_{3,3} \left[ \eta \bigg| \begin{array}{ccc} (1-\frac{1}{\nu}, \frac{1}{\nu}); & (1-\frac{\beta}{\nu}, \frac{\beta}{\nu}); & (\mu,0) \\ (\mu-1, 1); & (1-\frac{1}{\nu}, \frac{1}{\nu}) ;& (0,1) \end{array} \right]
\end{equation*}
where $\mathbf{G}^{\nu, \beta}_{\mu}(x)=\mathbf{g}^{\nu, \beta}_{\mu}(x,1)$. Thus, 
\begin{align*}
\mathbf{G}^{\nu, \beta}_{\mu}(x) = & H^{2,1}_{3,3} \left[ x \bigg| \begin{array}{ccc} (1-\frac{1}{\nu}, \frac{1}{\nu}); & (1-\frac{\beta}{\nu}, \frac{\beta}{\nu}); & (\mu,0) \\ (\mu-1, 1); & (1-\frac{1}{\nu}, \frac{1}{\nu}) ;& (0,1) \end{array} \right]\\
= & \frac{1}{x} H^{2,1}_{3,3} \left[ x \bigg| \begin{array}{ccc} (1, \frac{1}{\nu}); & (1, \frac{\beta}{\nu}); & (\mu,0) \\ (\mu , 1); & (1, \frac{1}{\nu}) ;& (1,1) \end{array} \right]
\end{align*}
where we used, for $c=1$, the property of the H functions
\begin{equation}
\H = \frac{1}{x^c} \; H^{m,n}_{p,q}\left[ x \bigg| \begin{array}{l} (a_i + c \alpha_i, \alpha_i)_{i=1, .. , p}\\ (b_j + c \beta_j, \beta_j)_{j=1, .. , q}  \end{array} \right] \label{propH2}
\end{equation}
for all $c \in \mathbb{R}$ (see \citet{MS73}). From \eqref{propM1}, by observing that
\begin{equation*}
\mathcal{M}\left[ \frac{1}{t^{\beta/\nu}} \mathbf{G}^{\nu, \beta}_{\mu}\left( \frac{\cdot}{t^{\beta/\nu}} \right) \right](\eta) = \mathcal{M}[\mathbf{G}^{\nu, \beta}_{\mu}(\cdot)](\eta) t^{\frac{\eta -1}{\nu}\beta}
\end{equation*} 
we obtain the claimed result.

\subsection{Proof of Theorem 5}
Hereafter, we extend the result given in \cite{Dov2} (Lemmas \ref{lemmaU} and \ref{LemmaD}) and show how the Mellin convolution turns out to be useful in order to explicitly write the distributions of both stable subordinators and their inverse processes. Let us consider the time-stretching functions $\psi_m(s) = m \, s^{1/m}$, $s \in (0, \infty)$, $m \in \mathbb{N}$ and $\varphi_m$ such that $\psi_m = \varphi_m^{-1}$ (the inverse function of $\varphi_m$).
\begin{lem}{\cite[Lemma 2]{Dov2}}
The Mellin convolution $e^{\star n}_{\boldsymbol{\mu}}(x, \varphi_{n+1}(t))$ where $\mu_j = j \, \nu$, for $j=1,2,\ldots, n$ is the density law of a $\nu$-stable subordinator $\{ H^{(\nu)}_t,\, t>0 \}$ with $\nu=1/(n+1)$, $n \in \mathbb{N}$. Thus, we have
\begin{equation*}
h_{\nu}(x,t) = e^{\star n}_{\boldsymbol{\mu}}(x, \varphi_{n+1}(t)), \quad x,t>0.
\end{equation*}
\label{lemmaU}
\end{lem}
We recall that $e_\mu=g^{-1}_{\mu}$ is the $1$-dimensional law of $E_\mu$.
\begin{lem}{\cite[Lemma 3]{Dov2}} 
The Mellin convolution $g^{(n+1), \star n}_{\boldsymbol{\mu}}(x, \psi_{n+1}(t))$ where $\mu_j  = j\, \nu$, $j=1,2,\ldots ,  n$ and $\nu=1/(n+1)$, $n \in \mathbb{N}$, is the density law of a $\nu$-inverse process $\{L^{(\nu)}_t, \, t>0 \}$. Thus, we have
\begin{equation*}
l_{\nu}(x,t) = g^{(n+1), \star n}_{\boldsymbol{\mu}}(x, \psi_{n+1}(t)), \quad x,t>0.
\end{equation*}
\label{LemmaD}
\end{lem}
We observe that $g^{n+1}_{\mu}$ is the $1$-dimensional law of $G_\mu^{\frac{1}{n+1}}(t^{n+1}) = G_\mu^\nu(t^{1/\nu})$.\\

The following facts will be useful later on.
\begin{lem}
For $g^{\gamma}_{\mu}=g^{\gamma}_{\mu}(x,t)$, $x \in \mathbb{R}_{+}$, $t>0$, $\mu>0$ the following hold:
\begin{itemize}
\item [i)] $\star$-commutativity: $g^{\gamma_1}_{\mu_1} \star g^{\gamma_2}_{\mu_2} = g^{\gamma_2}_{\mu_2} \star g^{\gamma_1}_{\mu_1}$ for all $\gamma_1, \gamma_2 \neq 0$.
\item [ii)] $*$-commutativity: $g^{\gamma_1}_{\mu_1} * g^{\gamma_2}_{\mu_2} = g^{\gamma_2}_{\mu_2} * g^{\gamma_1}_{\mu_1}$ for all $\gamma_1, \gamma_2 \neq 0$. Furthermore, for $\gamma=1,2$,
\[g^{\gamma}_{\mu_1} * g^{\gamma}_{\mu_2} = g^{\gamma}_{\mu_1 + \mu_2}. \]
\item [iii)] ($\star$,$*$)-distributivity: when $\star$- and $*$- commutativity hold, we have 
\begin{equation*}
g^{\gamma}_{\mu_1} \star (g^{\gamma}_{\mu_2} * g^{\gamma}_{\mu_3}) = (g^{\gamma}_{\mu_1} \star g^{\gamma}_{\mu_2}) * (g^{\gamma}_{\mu_1} \star g^{\gamma}_{\mu_3}).
\end{equation*}
\item [iv)] for $\mu_1, \mu_2, c \in \mathbb{N}$
\begin{equation}
g^{1}_{\mu_1 \cdot \mu_2} = \divideontimes_{j_1=1}^{\mu_1} g^{1}_{\mu_2} = \divideontimes_{j_2=1}^{\mu_2} g^{1}_{\mu_1} \quad \textrm{ and } \quad  \divideontimes_{j_1=1}^{\mu_1 \pm c} g^{1}_{\mu_2} =  g^{1}_{\mu_1 \cdot \mu_2 \cdot c^{\pm 1}}
\end{equation}
where $\divideontimes_{j=1}^{n} f = f_1 * f_2 * , \ldots , * f_n$.
\end{itemize}
\label{LemmaCONV}
\end{lem}
\begin{proof} The point $i)$ comes directly from the formula \eqref{ProdVarM} and the fact that $g^{\gamma_j}_{\mu_j} \in \mathbb{F}_{1}$, $\forall \gamma_j \neq 0$, and $\mu_j>0$, $j=1,2$. We show that $ii)$ holds. For $\gamma=1$, $\forall\, t>0$, $g^{\gamma}_{\mu}$ is the gamma density with Laplace transform $\mathcal{L}[g^{1}_{\mu}(\cdot, t)](\lambda) = 1/(1+\lambda t)^{\mu}$ and the statement follows easily. This is a well-known result. The case $\gamma=2$ is considered in \citet{SW73} being $g^2_\mu$ the semigroup for a Bessel process $R_{2\mu} = G^{1/2}_{2\mu}$ where $G_\mu$ satisfies the stochastic equation \eqref{SDEBes}. The result in $iii)$ can be obtained by considering, $\forall\, t>0$, the independent r.v.'s $G^{\gamma_j}_{\mu_j}(t)$, $j=1,2,3$ with densities $g^{\gamma_j}_{\mu_j}=g^{\gamma_j}_{\mu_j}(x,t)$, $j=1,2,3$, $x \in \mathbb{R}_{+}$. From the fact that $g^{\gamma_j}_{\mu_j} \in \mathbb{F}_{1}$, $j=1,2,3$ we have that $G^{\gamma_1}_{\mu_1}(G^{\gamma_2}_{\mu_2}(t)) \stackrel{law}{=}G^{\gamma_2}_{\mu_2}(G^{\gamma_1}_{\mu_1}(t))$, that is, $\forall\, j$, $g^{\gamma_j}_{\mu_j}$ are commutative under $\star$. For this reason and the $*$-commutativity, $\forall\, t>0$, we can write
\begin{align*}
G^{1}_{\mu_1}(G^{1}_{\mu_2}(t) + G^{1}_{\mu_3}(t)) = & G^{1}_{\mu_1}(G^{1}_{\mu_2 + \mu_3}(t)) = G^{1}_{\mu_2 + \mu_3}(G^{1}_{\mu_1}(t))\\
= & G^{1}_{\mu_2}(G^{1}_{\mu_1}(t)) + G^{1}_{\mu_3}(G^{1}_{\mu_1}(t)) .
\end{align*}
In the last calculation we have used the fact that
\begin{equation*}
E \Big[ \exp\left( - \lambda [ G_{\mu_1}(s) + G_{\mu_2}(s) ] \right) \Big| s=X_t \Big]= E \Big[ \exp\left( - \lambda G_{\mu_1  + \mu_2}(s) \right) \Big| s=X_t \Big].
\end{equation*} 
The same result can be achieved for $\gamma=2$. In order to prove $iv)$ we proceed as follows: first of all we observe that $ii)$ implies $g^{1}_{\mu_1 \cdot \mu_2} = \divideontimes_{j_1=1}^{\mu_1} g^{1}_{\mu_2} = \divideontimes_{j_2=1}^{\mu_2} g^{1}_{\mu_1}$. Second of all we consider that 
\[  \divideontimes_{j_1=1}^{\mu_1 + c} g^{1}_{\mu_2} = \divideontimes_{j_c =1}^{c} \divideontimes_{j_1=1}^{\mu_1} g^{1}_{\mu_2}=  g^{1}_{\mu_1 \cdot \mu_2 \cdot c} \] 
whereas 
\[  \divideontimes_{j_c=1}^{c}\divideontimes_{j_1=1}^{\mu_1 - c} g^{1}_{\mu_2} =  \divideontimes_{j_c=1}^{c} \, g^{1}_{\frac{\mu_1 \cdot \mu_2}{c}} = g^{1}_{\mu_1 \cdot \mu_2} \] 
and this concludes the proof.
\end{proof}

\begin{prop}
The following holds true
\begin{equation}
g^{\gamma, \star n}_{\boldsymbol{\mu}}(x, t) = \tilde{g}^{\gamma}_{\mu_j} \circ g^{1, \star (n-1)}_{\boldsymbol{\mu} \setminus \{ \mu_{j} \}} (x, t^{\gamma}), \quad \forall \mu_{j} \in \boldsymbol{\mu}, \; j=1,2, \ldots , n \label{differentConv}
\end{equation}
where
\begin{equation*}
f_1 \circ f_2(x,t) = \int_0^\infty f_1(x, s)\, f_2(s, t)\, ds
\end{equation*}
for $f_j : [0, +\infty) \mapsto [0, +\infty)$, $j=1,2$.
\label{propCirc}
\end{prop}
\begin{proof} 
Fix $n=3$. $\forall\, t>0$, it is enough to consider the r.v.'s $G^{\gamma}_{\mu}$, $\tilde{G}^{\gamma}_{\mu}$ and their density laws $g^{\gamma}_{\mu}$,  $\tilde{g}^{\gamma}_{\mu}$ where $\tilde{g}^{\gamma}_{\mu}(x,t) = g^{\gamma}_{\mu}(x, t^{1/\gamma})$ or equivalently $G^{\gamma}_{\mu}(t) \stackrel{law}{=} \tilde{G}^{\gamma}_{\mu}(t^{\gamma})$. In this setting, we have that $X(t)=G^{\gamma}_{\mu_1}(G^{\gamma}_{\mu_2}(G^{\gamma}_{\mu_3}(t)))$ can be written as $X(t) \stackrel{law}{=} \tilde{G}^{\gamma}_{\mu_1}(G^{1}_{\mu_2}(G^{1}_{\mu_3}(t^{\gamma})))$ thank to the fact that $\left( G^{\gamma}_{\mu} \right)^{\gamma} \stackrel{law}{=} G^{1}_{\mu}$. Thus, we can write
\begin{equation*}
g^{\gamma, \star 3}_{\boldsymbol{\mu}}(x,t) = \tilde{g}^{\gamma}_{\mu_1} \circ g^{1, \star 2}_{(\mu_2, \mu_3)}(x, t^{\gamma}) = \tilde{g}^{\gamma}_{\mu_1} \circ (g^{1}_{\mu_2} \star g^{1}_{\mu_3})(x, t^{\gamma}).
\end{equation*}
Thanks to the $\star$-commutativity we have that $g^1_{\mu_2} \star g^1_{\mu_3} = g^1_{\mu_3} \star g^1_{\mu_2}$ and also that
\begin{equation*}
g^{\gamma, \star 3}_{\boldsymbol{\mu}}(x,t) = \tilde{g}^{\gamma}_{\mu_2} \circ g^{1, \star 2}_{(\mu_1, \mu_3)}(x, t^{\gamma}) = \tilde{g}^{\gamma}_{\mu_2} \circ (g^{1}_{\mu_1} \star g^{1}_{\mu_3})(x, t^{\gamma}).
\end{equation*}
By considering $n$ processes, the formula \eqref{differentConv} immediately appears.
\end{proof}

\begin{os}
\label{remark1quinto}
\normalfont
For $\nu=1/5$, we have that
\begin{equation}
l_{1/5}(x, t) = g^{5, \star 4}_{\boldsymbol{\mu}_1}(x, 5t^{1/5}), \quad \boldsymbol{\mu}_1 = (1/5, 2/5, 3/5, 4/5)
\end{equation}
and, from the Proposition \ref{propCirc} and the Lemma \ref{LemmaCONV},
\begin{align*}
& g^{5, \star 4}_{(1/5, 2/5, 3/5, 4/5)} =  \tilde{g}^{5}_{1/5} \circ g^{1, \star 3}_{(4/5, 3/5, 2/5)}\\
= & \tilde{g}^{5}_{1/5} \circ \left[ (g^{1}_{1/5} * g^{1}_{1/5}* g^{1}_{1/5} * g^{1}_{1/5}) \star (g^{1}_{1/5} * g^{1}_{1/5} * g^{1}_{1/5}) \star g^{1}_{2/5}\right]\\
= & \tilde{g}^{5}_{1/5} \circ \left[ (g^{1}_{1/5} * g^{1}_{1/5}* g^{1}_{1/5} * g^{1}_{1/5}) \star (g^{1, \star 2}_{(2/5, 1/5)} * g^{1, \star 2}_{(2/5, 1/5)} * g^{1, \star 2}_{(2/5, 1/5)}) \right]\\
= & \tilde{g}^{5}_{1/5} \circ \left[ \divideontimes_{k=1}^{12} \left( g^{1, \star 3}_{(2/5, 1/5, 1/5)}\right)_k \, \right]\\
= & \tilde{g}^{5}_{1/5} \circ g^{1, \star 3}_{24/5, 1/5, 1/5} (x, 5^5t) = g^{5, \star 4}_{(24/5, 1/5, 1/5, 1/5)}
\end{align*}
where
\begin{equation*}
\boldsymbol{\mu}_2 = (24/5, 1/5, 1/5, 1/5)
\end{equation*}
and $\boldsymbol{\mu}_1, \boldsymbol{\mu}_2 \in \mathscr{P}_{5}^{4}\left( 24 \right)$.
Finally, we obtain
\begin{align*}
g^{5, \star 4}_{(24/5, 1/5, 1/5, 1/5)}(x,t) = & g^{5, \star 2}_{(24/5, 1/5)} \star g^{5, \star 2}_{(1/5, 1/5)}(x, t).
\end{align*}
From \eqref{eConv} the corresponding integral representation emerges.
\end{os}

\begin{te}
Fix $\boldsymbol{\mu} \in \mathscr{P}_{\kappa}^{n}\left( \varrho \right)$. Then $g^{1, \star n}_{\boldsymbol{\mu}}= g^{1, \star n}_{\boldsymbol{\vartheta}}$ for all $\boldsymbol{\vartheta} \in \mathscr{P}_{\kappa}^{n}\left( \varrho \right)$. \label{teoFat}
\end{te}
\begin{proof}
Fix $\kappa$, $\varrho \in \mathbb{N}$. We have $g^{1, \star n}_{\boldsymbol{\mu}} = g^{1}_{\mu_1} \star \ldots \star g^{1}_{\mu_n}$, $\boldsymbol{\mu}=(\mu_1, \ldots \mu_n) \in \mathscr{P}^n_{\kappa}(\varrho)$. From \eqref{setS} we can write $\boldsymbol{\mu}=\frac{1}{\kappa}(\tilde{\mu}_1, \ldots , \tilde{\mu}_n)$. Let us first consider $n=2$. We recall that $g^{1,\star 2}_{(\mu_1, \mu_2)} = g^{1,\star 2}_{(\mu_2, \mu_1)}$ from the $\star$-commutativity. Thus, form the properties $i)$ and $ii)$ of the Lemma \ref{LemmaCONV}, we have that
\begin{align*}
g^{1, \star 2}_{(\frac{1}{\kappa}\tilde{\mu}_1, \frac{1}{\kappa}\tilde{\mu}_2)} = & \divideontimes_{j_1=1}^{\tilde{\mu}_1} g^{1, \star 2}_{(\frac{1}{\kappa}, \frac{1}{\kappa}\tilde{\mu}_2)} = \divideontimes_{j_1=1}^{\tilde{\mu}_1} g^{1, \star 2}_{(\frac{1}{\kappa}\tilde{\mu}_2, \frac{1}{\kappa})}\\
= & \divideontimes_{j_1=1}^{\tilde{\mu}_1} \divideontimes_{j_2=1}^{\tilde{\mu}_2} g^{1, \star 2}_{(\frac{1}{\kappa}, \frac{1}{\kappa})} = \divideontimes_{j_1=1}^{\tilde{\mu}_1} \divideontimes_{j_2=1}^{\tilde{\mu}_2} g^{1, \star 2}_{\frac{1}{\kappa}(1, 1)}.
\end{align*}
For $n \in \mathbb{N}$, $\boldsymbol{\vartheta}_{0}=\frac{1}{\kappa} (1, \ldots , 1) \in \mathbb{R}^{n}_{+}$, we can write
\begin{equation*}
g^{1, \star n}_{\boldsymbol{\mu}} = g^{1, \star n}_{\frac{1}{\kappa}(\tilde{\mu}_1, \ldots , \tilde{\mu}_n)} = \divideontimes_{j_1=1}^{\tilde{\mu}_1} \ldots \divideontimes_{j_n=1}^{\tilde{\mu}_n} g^{1, \star n}_{\boldsymbol{\vartheta}_{0}}.
\end{equation*}
We shall refer to $\boldsymbol{\vartheta}_0$ as the $0$-configuration. We first observe that
\begin{equation*}
g^{1, \star n}_{\boldsymbol{\mu}} = \divideontimes_{j_1=1}^{\tilde{\mu}_1} \ldots \divideontimes_{j_n=1}^{\tilde{\mu}_n} g^{1, \star n}_{\boldsymbol{\vartheta}_{0}} = g^{1, \star n}_{\frac{1}{\kappa}(\varrho, 1, \ldots , 1)}, \quad \varrho= \prod_{j=1}^n \tilde{\mu}_j ,
\end{equation*}
or equivalently
\begin{align*}
g^{1, \star n}_{\boldsymbol{\mu}} = & \divideontimes_{j_1=1}^{\tilde{\mu}_1} \ldots \divideontimes_{j_n=1}^{\tilde{\mu}_n} g^{1, \star n}_{\boldsymbol{\vartheta}_{0}} = \divideontimes_{j_i=1}^{\tilde{\mu}_i} g^{1, \star n}_{\frac{1}{\kappa}( \varrho / \tilde{\mu}_{i}, 1, \ldots , 1)} = g^{1, \star n}_{\frac{1}{\kappa}( \varrho_i, \tilde{\mu}_{i}, 1, \ldots , 1)}
\end{align*}
$\varrho_i = \varrho / \tilde{\mu}_{i}$, for all $i=1,2, \ldots , n$. The last identity comes from the $\star$-commutativity. By exploiting the $*$-commutativity and the $\star$-commutativity we have that $g^{1, \star n}_{\boldsymbol{\mu}} = g^{1, \star n}_{\boldsymbol{\theta}}$ for all $\boldsymbol{\theta}$ such that
\begin{align*}
\mathbb{R}^n_{+} \ni \boldsymbol{\theta} = \frac{1}{\kappa} (\varrho_{\mathbf{i}}, \tilde{\mu}_{\mathbf{i}}, \mathbf{1}), \quad \varrho_{\mathbf{i}}=\varrho / \prod_{s_j =1}^{|\mathbf{i}|} \tilde{\mu}_{s_j}
\end{align*}
where $dim(\mathbf{1})=n - |\mathbf{i}| -1$, $\tilde{\mu}_{\mathbf{i}} = (\tilde{\mu}_{s_1}, \ldots , \tilde{\mu}_{s_{|\mathbf{i}|}}) \in \mathbb{N}^{|\mathbf{i}|}_{+}$, $s_j \in \mathbf{i}$, $j=1,2, \ldots |\mathbf{i}|$ and $|\mathbf{i}| < n$ is the cardinality of $\mathbf{i}$. A further configuration is given by $\bar{\boldsymbol{\theta}} = (\varrho_{\mathbf{i}}, \tilde{\mu}_{\mathbf{i}}) / \kappa$ where $|\mathbf{i}| = n-1$. In this case, $\bar{\boldsymbol{\theta}} \in \mathscr{P}^{n}_{\kappa}(\varrho)$ is obtainable by $n!$ permutation of the elements of $\boldsymbol{\mu}$. In a more general setting, for $\bar{\alpha}=(\alpha_1 \cdot \beta, \alpha_2, \ldots , \alpha_n) \in \mathbb{N}^n$, $\beta \in \mathbb{N}$, $c \in \mathbb{N}$ the following rules hold
\begin{equation}
g^{1, \star n}_{\bar{\alpha}} = \divideontimes_{j=1}^{\beta} g^{1, \star n}_{(\alpha_1, \alpha_2, \ldots , \alpha_n)}
\end{equation}
(see $iv)$, Lemma \ref{LemmaCONV}) and, for $c \geq 1$,
\begin{equation}
\divideontimes_{j=1}^{\beta \pm c} g^{1, \star n}_{(\alpha_1, \alpha_2, \ldots , \alpha_n)} = \divideontimes_{j=1}^{\beta \pm c} g^{1, \star n}_{(\alpha_{\sigma_1}, \alpha_{\sigma_2}, \ldots , \alpha_{\sigma_n})}= g^{1, \star n}_{(\alpha_{\sigma_1} \cdot \beta \cdot c^{\pm 1}, \alpha_{\sigma_2}, \ldots , \alpha_{\sigma_n})}\label{exampAst}
\end{equation}
for all permutation of $\{ \sigma_{j} \}$, $j=1,2, \ldots , n$. We recall that 
\begin{equation*}
\divideontimes_{j=1}^{\beta + c} g_{\alpha}= \divideontimes_{j_1=1}^{c}\divideontimes_{j_2=1}^{\beta} g_{\alpha} = g_{\alpha \beta c},
\end{equation*} 
for $\alpha, \beta, c \in \mathbb{N}$. By making use of the properties $i)$, $ii)$ and $iii)$ of the Lemma \ref{LemmaCONV} we can obtain all possible configurations of $\boldsymbol{\vartheta} \in \mathbb{R}^n_+$ starting from the $0$-configuration $\boldsymbol{\vartheta}_{0}$. All different configurations of $\boldsymbol{\vartheta}$ are included in $\mathscr{P}^{n}_{\kappa}(\varrho)$. From \eqref{exampAst}, for all $\boldsymbol{\vartheta}=(\vartheta_1, \ldots , \vartheta_n) \in \mathscr{P}^{n}_{\kappa}(\varrho)$ we have that $\prod_{j=1}^{n} \vartheta_j = \varrho$. This concludes the proof.
\end{proof}

By collecting all pieces together we obtain the claimed result.

\appendix

\section{Fox functions and Mellin transform}
\label{appeA}
The Mellin transform of $f \in \mathbb{M}_{a}^{b}$ (see Definition \ref{defM}) is defined as
\begin{equation}
\mathcal{M}[f(\cdot)](\eta) = \int_0^{\infty} x^{\eta -1} f(x) dx, \quad \eta \in \mathbb{H}_{a}^{b}. \label{mellinTRF}
\end{equation}
Let us point out some useful operational rules that will be useful throughout the paper: for some $-\infty < a < b < \infty$ and $\mathfrak{b} >0$, $f,f_1, f_2 \in \mathbb{M}^{b}_{a}$:
\begin{align}
\int_0^\infty x^{\eta -1} f(\mathfrak{b} x) dx = & \mathfrak{b}^{-\eta} \mathcal{M}\left[ f(\cdot) \right] (\eta), \label{propM1}\\
\mathcal{M}\left[ x^{\mathfrak{b}} f(\cdot) \right] (\eta) = & \mathcal{M}\left[ f(\cdot) \right] (\eta + \mathfrak{b}), \label{propM2}\\
\mathcal{M}\left[ \int_{0}^{\infty} f_1\left( \frac{\cdot}{s} \right) f_2\left( s \right) \frac{ds}{s} \right] (\eta) = & \mathcal{M}\left[ f_1(\cdot) \right] (\eta) \times \mathcal{M}\left[ f_2(\cdot) \right] (\eta), \label{conv}\\
\mathcal{M}[I(\cdot)](\eta) = & \eta^{-1} \mathcal{M}[f(\cdot)](\eta +1), \label{melIntegral}
\end{align}
where
\begin{equation}
I(x) = \int_{x}^{\infty} f(s)ds, \quad x>0,
\end{equation}
see e.g.\ \citet{GPSko06}. The formula \eqref{conv} is the well-known Mellin convolution formula which turns out to be useful in the study of the product of random variables. \\

We say that $f \in \widetilde{\mathbb{M}}_{n}$ if $f \in \mathbb{M}_{a}^{b}$ and is a rapidly decreasing function such that 
\begin{equation*}
\exists \mathfrak{a} \in \mathbb{R}\; s.t.\; \lim_{x \to +\infty} x^{\mathfrak{a} - k - 1} \frac{d^k f}{d x^k} (x) = 0, \quad k=0,1,\ldots , n, \quad n \in \mathbb{N}, \; x \in \mathbb{R}_+ 
\end{equation*}
and
\begin{equation*}
\exists \mathfrak{b} \in \mathbb{R}\; s.t.\; \lim_{x \to 0^+} x^{\mathfrak{b} - k - 1} \frac{d^k f}{d x^k} (x) = 0, \quad k=0,1,\ldots , n, \quad n \in \mathbb{N}, \; x \in \mathbb{R}_+. 
\end{equation*}
For $f \in \widetilde{\mathbb{M}}_{n-1}$ and $n \in \mathbb{N}$ we have that
\begin{align}
\mathcal{M} \left[ \frac{d^n f}{d x^n}(\cdot) \right] (\eta) = & (-1)^n \frac{\Gamma(\eta)}{\Gamma(\eta - n)} \mathcal{M} \left[ f(\cdot) \right] (\eta - n) \label{derMint0} \\
= & \frac{\Gamma(1 + n -\eta )}{\Gamma(1-\eta)} \mathcal{M} \left[ f(\cdot) \right] (\eta - n).  \label{derMint}
\end{align}

As a generalized version of the integer derivatives \eqref{derMint0} and \eqref{derMint} we introduce the Mellin transform of fractional derivatives \eqref{Rfracderr} and \eqref{Rfracder} (see \citet{KST06, SKM93} for details). For a given $f \in \mathbb{M}_{a}^{b}$ and $0 < \alpha < 1$, if $\Re\{ \eta \} >0$,
\begin{align}
\mathcal{M}\left[\frac{d^\alpha f}{d (-x)^\alpha}(\cdot)\right](\eta) = \frac{\Gamma(\eta)}{\Gamma(\eta -\alpha)} \mathcal{M}[f(\cdot)](\eta -\alpha) + (\mathcal{T}I_{0-}^{1-\alpha}f)(\eta) \label{derMfracmeno}
\end{align}
whereas, if $\Re\{ \eta \} < \alpha +1$,
\begin{equation}
\mathcal{M}\left[\frac{\partial^\alpha f}{\partial x^\alpha}(\cdot)\right](\eta) = \frac{\Gamma(1+ \alpha - \eta)}{\Gamma(1-\eta)}\mathcal{M}\left[ f(\cdot) \right](\eta -\alpha) + (\mathcal{T}I_{0+}^{1-\alpha}f)(\eta) \label{derMfrac2}
\end{equation}
where
\begin{equation*}
(\mathcal{T} I_{0\pm}^{1-\alpha}f)(\eta) =  \frac{\Gamma(\alpha - \eta)}{\Gamma(1-\eta)} \left[ x^{\eta - 1} (I^{1-\alpha}_{0\pm}f)(x) \right]_{x=0}^{x=\infty}
\end{equation*}
and
\begin{equation*}
(I_{0-}^{1-\alpha} f)(x) = \frac{1}{\Gamma(\alpha)} \int_x^\infty (s-x)^{\alpha -1} f(s)\, ds,\quad  x>0, 
\end{equation*}
\begin{equation*}
(I^{1-\alpha}_{0+}f)(x) = \frac{1}{\Gamma(\alpha)} \int_0^x (x-s)^{\alpha -1} f(s)\, ds,\quad  x>0
\end{equation*}
are the right and left fractional integrals.\\

The Fox functions, also referred to as Fox's H-functions, H-functions, generalized Mellin-Barnes functions, or generalized Meijer's G-functions, were introduced by Fox \cite{FOX61} in 1996. Here, the Fox's H-functions will be recalled as the class of functions uniquely identified by their Mellin transforms. A function $f \in \mathbb{M}_{a}^{b}$ can be written in terms of H-functions by observing that
\begin{equation}
\int_{0}^{\infty} x^{\eta } H^{m,n}_{p,q}\left[ x \bigg| \begin{array}{l} (a_i, \alpha_i)_{i=1, .. , p}\\ (b_j, \beta_j)_{j=1, .. , q}  \end{array} \right] \frac{dx}{x} = \mathcal{M}^{m,n}_{p,q}\left[ \eta \bigg| \begin{array}{l} (a_i, \alpha_i)_{i=1, .. , p}\\ (b_j, \beta_j)_{j=1, .. , q}  \end{array} \right] 
\label{intStrip}
\end{equation}
with $\eta \in \mathbb{H}_{a}^{b}$ where
\begin{equation}
\mathcal{M}^{m,n}_{p,q} \left[ \eta \bigg| \begin{array}{l} (a_i, \alpha_i)_{i=1, .. , p}\\ (b_j, \beta_j)_{j=1, .. , q}  \end{array} \right]  = \frac{\prod_{j=1}^{m} \Gamma(b_j + \eta \beta_j) \prod_{i=1}^{n} \Gamma(1-a_i - \eta \alpha_i)}{\prod_{j=m+1}^{q} \Gamma(1-b_j - \eta \beta_j) \prod_{i=n+1}^p \Gamma(a_i + \eta \alpha_i)}.
\label{mellinHfox}
\end{equation}
Thus, according to a standard notation, the Fox H-function can be defined as follows
\begin{equation}
\H = \frac{1}{2\pi i}\int_{\mathcal{P}(\mathbb{H}_{a}^{b})} \mathcal{M}^{m,n}_{p,q}(\eta) x^{-\eta} d\eta \label{defHFoxApp}
\end{equation}
where $\mathcal{P}(\mathbb{H}_{a}^{b})$ is a suitable path in the complex plane $\mathbb{C}$ depending on the fundamental strip ($\mathbb{H}_{a}^{b}$) such that the integral \eqref{intStrip} converges. For an extensive discussion on this functions see \citet{FOX61, KST06, MS73}.

\end{document}